\documentclass[11pt]{article}
\usepackage{geometry}
\usepackage{amsmath,amsfonts,amssymb,amsthm}
\usepackage{graphicx}
\usepackage{enumerate,enumitem}
\usepackage{bbm}
\usepackage{verbatim}
\usepackage{caption}
\usepackage{subcaption}
\usepackage{hyperref,color}
\usepackage[capitalize,nameinlink]{cleveref}
\usepackage[dvipsnames]{xcolor}
\hypersetup{
	colorlinks=true,
	pdfpagemode=UseNone,
	citecolor=OliveGreen,
	linkcolor=NavyBlue,
	urlcolor=Magenta,
	pdfstartview=FitW`
}
\usepackage{appendix}
\crefname{appsec}{Appendix}{Appendices}
\usepackage{tikz}

\newtheorem{theorem}{Theorem}
\newtheorem{proposition}[theorem]{Proposition}
\newtheorem{lemma}[theorem]{Lemma}

\newtheorem{remark}[theorem]{Remark}

\newtheorem{claim}[theorem]{Claim}

\crefname{lem}{Lemma}{Lemmas}
\crefname{thm}{Theorem}{Theorems}

\theoremstyle{definition}
\newtheorem{definition}[theorem]{Definition}

\crefname{defn}{Definition}{Definitions}

\newcommand{\E}{\mathbb{E}}
\newcommand{\Var}{\mathrm{Var}}

\newcommand{\ip}[2]{\left\langle #1 , #2 \right\rangle}

\newcommand{\poly}{\mathrm{poly}}
\newcommand{\dist}{\mathrm{dist}}

\newcommand{\Ex}{\mathbb{E}}

\newcommand{\eps}{\varepsilon}

\newcommand{\integers}{\mathbb Z}

\newcommand{\bs}{\backslash}

\newcommand{\N}{\mathbb{N}}

\newcommand{\R}{\mathbb{R}}

\newcommand{\V}{\mathbb{V}}
\newcommand{\Edges}{\mathbb{E}}
\newcommand{\T}{\mathbb{T}}
\newcommand{\Z}{\mathbb{Z}}
\newcommand{\D}{\mathcal{E}}

\newcommand{\EE}{\mathcal{E}}

\newcommand{\VM}{\mathrm{VM}}
\newcommand{\EM}{\mathrm{EM}}
\newcommand{\PVM}{\mathrm{PVM}}

\newcommand{\GVM}{\mathrm{GVM}}

\newcommand{\GAP}{\mathsf{gap}}

\newcommand{\Ptb}{P_{\textsc{tb}}}

\newcommand{\allone}{\boldsymbol{1}}

\newcommand{\SWTC}[2]{\boldsymbol{SW}_{#1}^{#2}}
\newcommand{\SWT}{{\boldsymbol{SW}_\mathcal{D}}}

\newcommand{\PSW}{\boldsymbol{SW}}
\newcommand{\PSWE}{{\widetilde P}_{\textsc{sw}}}

\newcommand{\PHB}{P_{\textsc{hb}}}
\newcommand{\PSB}{P_{\textsc{sb}}}

\newcommand{\inner}[3]{\langle #1 , #2 \rangle_{#3}}

\DeclareMathOperator*{\JC}{\Omega_\textsc{joint}}
\DeclareMathOperator*{\JCH}{{\hat \Omega}_\textsc{joint}}
\DeclareMathOperator*{\RC}{\Omega_\mathrm{RC}}

\DeclareMathOperator*{\1}{\mathbbm{1}}
\newcommand{\var}[2]{\mathrm{Var}_{#1}^{#2}}
\DeclareMathOperator{\Ent}{Ent}

\newcommand{\G}{\hat{G}}
\newcommand{\MC}{\mathcal{\hat{M}}}
\newcommand{\p}{\hat{p}}
\newcommand{\df}[2]{\mathcal{E}_{#1}(#2,#2)}
\newcommand{\M}{M}
\def\QMHB{Q_\textsc{mhb}}

\def\PHB{P_\textsc{hb}}
\def\PMHB{P_\textsc{mhb}}

\newcommand{\PSWI}{P^{(i)}_{\textsc{sw}}}

\def\tmix{\tau_{\rm mix}}

\title{The Swendsen-Wang Dynamics on Trees}
\author{
	Antonio Blanca\thanks{Pennsylvania State University.
		Email: ablanca@cse.psu.edu.
		Research supported in part by NSF grant CCF-1850443.}
	\and
	Zongchen Chen\thanks{School of Computer Science, Georgia Institute of Technology, Atlanta, GA 30332.
		Email: \{chenzongchen, vigoda\}@gatech.edu.
		Research supported in part by NSF grant CCF-2007022.}
	\and
	Daniel \v{S}tefankovi\v{c}\thanks{Department of Computer Science, University of Rochester, Rochester, NY 14627.
		Email: stefanko@cs.rochester.edu.
		Research supported in part by NSF grant CCF-2007287.}
	\and
	Eric Vigoda$^\dag$
}
\date{\today}

\begin{document}

\maketitle

\begin{abstract}
The Swendsen-Wang algorithm is a sophisticated, widely-used Markov chain for
sampling from the Gibbs distribution for the ferromagnetic Ising and Potts models.
This chain has proved difficult to analyze,
due in part to the global nature of its updates.
We present optimal bounds on the convergence rate of the Swendsen-Wang algorithm
for the complete $d$-ary tree. Our bounds extend to the non-uniqueness region and apply to all boundary conditions. 

We show 
that the spatial mixing conditions known as \emph{Variance Mixing} and \emph{Entropy Mixing},
introduced 
in the study of local Markov chains by Martinelli et al.\ (2003),
imply $\Omega(1)$
spectral gap and $O(\log{n})$ mixing time, respectively, for the Swendsen-Wang dynamics on the $d$-ary tree.
We also show that these bounds are asymptotically optimal.
As a consequence, we establish $\Theta(\log{n})$ mixing for the Swendsen-Wang dynamics 
for \emph{all} boundary conditions throughout the tree uniqueness region;
in fact, our bounds hold beyond the uniqueness threshold for the Ising model, 
and for the $q$-state Potts model when $q$ is small with respect to $d$. 
Our proofs feature a novel spectral view of the Variance Mixing condition 
inspired by several recent
rapid mixing results on high-dimensional expanders
and utilize recent work on block factorization of entropy under spatial mixing conditions.
\end{abstract}

\section{Introduction}

Spin systems are idealized models of a physical system in equilibrium which are utilized in statistical physics to study phase transitions.
A phase transition occurs when there is a dramatic change in the macroscopic properties
of the system resulting from a small (infinitesimal in the limit) change in one of the parameters
defining the spin system.
The macroscopic properties of the system 
manifest with the persistence (or lack thereof) of long-range influences.
There is a well-established mathematical theory
connecting the absence of these influences to the
fast convergence of Markov chains. In this paper, we study this connection
on the regular tree, known as the \emph{Bethe lattice} in statistical physics~\cite{Bethe,Georgii}.

The most well-studied example of a spin system is the \emph{ferromagnetic $q$-state Potts model},
which contains the {\em Ising model} ($q = 2$) as a special case.
The Potts model is especially important 
as fascinating phase transitions (first-order vs.\ second-order)
are now understood rigorously in various contexts~\cite{BDC,DCST,DGHMT,CET,CDLLPS}.

Given a graph $G=(V,E)$,
configurations of the Potts model are 
assignments of spins~$[q] = \{1,2,\dots,q\}$ to the vertices of $G$.
The parameter $\beta>0$ (corresponding to the inverse of the temperature of
the system) controls the strength of nearest-neighbor interactions,
and the probability of a configuration $\sigma \in [q]^V$ in the \emph{Gibbs distribution} is such that \begin{equation}
\label{eq:Gibbs}
\mu(\sigma)=\mu_G(\sigma) = \frac{e^{-\beta |D(\sigma)|}}{Z}, 
\end{equation}
where $D(\sigma)=\{\{v,w\}\in E: \sigma(v)\neq\sigma(w)\}$ 
denotes the set of bi-chromatic edges in $\sigma$, and $Z$ is the normalizing constant known as the \emph{partition function}.

The {\em Glauber dynamics} is the simplest example of a Markov chain for sampling from the Gibbs distribution;
it updates the spin at a randomly chosen vertex in each step.
In many settings, as we detail below, the Glauber dynamics converges exponentially slow
at low temperatures (large $\beta$) due to the local nature of its transitions and
the long-range correlations in the Gibbs distribution.
Of particular interest are thus ``global''
Markov chains such as the \emph{Swendsen-Wang (SW) dynamics}~\cite{SW,ES},
which update a large fraction of the configuration in each step, thus
potentially overcoming the obstacles that hinder the
performance of the Glauber dynamics, and with steps that can be efficiently parallelized~\cite{AScc}.

The SW dynamics utilizes a close connection between
the Potts model and an alternative representation known as the \emph{random-cluster model}.  The
random-cluster model is defined on subsets of edges and is not a spin system as the
weight of a configuration depends on the global connectivity properties of the corresponding subgraph.
The transitions of the SW dynamics take a spin configuration, transform it to a ``joint'' spin-edge configuration, perform a step in the joint space, and then map
back to a Potts configuration. Formally, from a Potts configuration $\sigma_t\in [q]^V$,
a transition $\sigma_t\rightarrow\sigma_{t+1}$ is defined as follows:

\begin{enumerate}
	\item Let $M_t=M(\sigma_t)=E\setminus D(\sigma_t)$ denote the set of monochromatic edges in $\sigma_t$.
	\item Independently for each edge $e=\{v,w\}\in M_t$, keep $e$ with probability $p=1-\exp(-\beta)$
	and remove $e$ with probability $1-p$.  Let $A_t \subseteq M_t$ denote the resulting subset.
	\item In the subgraph $(V,A_t)$, independently for each connected component $C$ (including isolated vertices),
	choose a spin $s_C$ uniformly at random from $[q]$ and assign to each vertex in $C$ the spin $s_C$.
	This spin assignment defines $\sigma_{t+1}$.
\end{enumerate}

There are two standard measures of the convergence rate of a Markov chain.
The \emph{mixing time} is the number of steps to get within
total variation distance $\leq 1/4$ of its stationary distribution
from the worst starting state.  
The {\em relaxation time} is the inverse of the spectral gap
of the transition matrix of the chain 
and measures the speed of convergence from a ``warm start''.
For approximate counting algorithms the relaxation time 
is quite useful as it corresponds to the ``resample'' time~\cite{Gillman,KLS,JSV,JerrumBook}; see \cref{sec:prelim} for precise definitions 
and how these two notions relate to each other. 

There has been great progress in formally connecting phase transitions with 
the convergence rate of the Glauber dynamics.
Notably, for the $d$-dimensional integer lattice~$\integers^d$, 
a series of works established that a spatial mixing property known as \emph{strong spatial mixing (SSM)} implies
$O(n\log{n})$ mixing time of the Glauber dynamics~\cite{MO,Cesi,DSVW}.
Roughly speaking, SSM says that
correlations decay exponentially fast with the distance 
and is also known to imply optimal mixing and relaxation times of the SW dynamics on~$\integers^d$~\cite{BCSV,BCPSVstoc}.
These techniques utilizing SSM are particular to the lattice and 
do not extend to \emph{non-amenable} graphs (i.e., those whose boundary and volume are of the same order).  
The $d$-ary complete tree, which is the focus of this paper, is 
the prime example of a non-amenable graph.

On the regular $d$-ary tree, there are two fundamental phase transitions: the \emph{uniqueness} threshold $\beta_u$ and the \emph{reconstruction} threshold $\beta_r$. 
The smaller of these thresholds $\beta_u$ corresponds to
the uniqueness/non-uniqueness phase transition of the Gibbs measure on the infinite $d$-ary tree,
and captures whether the worst-case boundary configuration (i.e., a fixed configuration on the leaves of a finite tree)  
has an effect or not on the spin at the root (in the limit as the height of the tree grows).
The second threshold~$\beta_r$ is the reconstruction/non-reconstruction phase transition, marking the divide on 
whether or not a \emph{random} boundary condition (in expectation) affects the spin of the root.

There is a large body of work on the interplay between these phase transitions and the speed of convergence of the Glauber dynamics on the complete $d$-ary tree~\cite{MSW-Potts,MSW04,BKMP},
and more generally on bounded degree graphs~\cite{MS,GM,BGP}.
Our main contributions in this paper concern instead 
the speed of convergence of the SW dynamics on trees, how it is affected by these phase transitions, and the effects of the boundary condition.

Martinelli, Sinclair, and Weitz~\cite{MSW04,MSW-Potts} introduced a pair of spatial mixing (decay of correlation) conditions called 
\emph{Variance Mixing (VM)} and \emph{Entropy Mixing (EM)} 
which
capture the exponential decay of point-to-set correlations.
More formally, 
the VM and EM conditions hold
when there exist constants $\ell>0$ and $\varepsilon = \varepsilon(\ell)$
such that, for every vertex $v\in T$, 
the influence of the spin at $v$ on the spins of the vertices at distance $\ge \ell$ from $v$
in the subtree $T_v$ rooted at $v$ decays by a factor of at least $\eps$.  
For the case of VM, this decay of influence is captured  
in terms of the variance of any function $g$ that depends only on the spins of the vertices in $T_v$ at distance $\ge \ell$ from $v$; specifically, 
when conditioned on the spin at $v$, the conditional variance of $g$ 
is (on average) a factor $\eps$ smaller then the unconditional variance; see
Definition~\ref{def:vm} in Section~\ref{sec:variance} for the formal definition.
EM is defined analogously, with variance replaced by entropy; see Definition~\ref{def:em}.

It was established in~\cite{MSW04,MSW-Potts}
that VM and EM imply optimal bounds on the convergence rate of the Glauber dynamics on trees. We obtain optimal bounds for the speed of convergence of the SW dynamics under the same VM and EM spatial mixing conditions. 

\begin{theorem}
	\label{thm:intro-main}
	For all $q\geq 2$ and $d\geq 3$, for the $q$-state ferromagnetic Ising/Potts model on an $n$-vertex complete $d$-ary tree, Variance Mixing implies that the relaxation time of the Swendsen-Wang dynamics is $\Theta(1)$.
\end{theorem}

\begin{theorem}
	\label{thm:entropy-main}
	For all $q\geq 2$ and $d\geq 3$, for the $q$-state ferromagnetic Ising/Potts model on an $n$-vertex complete $d$-ary tree, Entropy Mixing implies that the mixing time of the Swendsen-Wang dynamics is $O(\log{n})$.
\end{theorem}


\noindent
The VM condition is strictly weaker (i.e., easier to satisfy) than the EM condition, but the relaxation time bound in~\Cref{thm:intro-main} is weaker than the mixing time bound in~\Cref{thm:entropy-main}. We also show that the mixing time in~\Cref{thm:entropy-main} is asymptotically the best possible.

\begin{theorem}
	\label{thm:lb:intro}
	For all $q\geq 2$, $d\geq 3$ and any $\beta > 0$,
	the mixing time of the SW dynamics on an 
	$n$-vertex complete $d$-ary tree is $\Omega(\log n)$ for any boundary condition.
\end{theorem}

We remark that the mixing time lower bound in \Cref{thm:lb:intro} applies to all inverse temperatures $\beta$ and all boundary conditions.

The VM
and EM conditions are properties of the Gibbs distribution induced by \emph{a specific} boundary condition on the leaves of the tree; 
this contrasts with other standard notions of decay of correlations such as SSM on $\Z^d$. 
This makes these  conditions quite suitable for understanding the speed of convergence of Markov chains under different boundary conditions.
For instance,~\cite{MSW04,MSW-Potts} established VM and EM for all boundary conditions provided $\beta < \max\{\beta_u,\frac{1}{2}\ln(\frac{\sqrt{d}+1}{\sqrt{d}-1})\}$ 
and for the monochromatic (e.g., all-red) boundary condition for all $\beta$. Consequently, we obtain the following results.

\begin{theorem}
	\label{thm:cor}
	For all $q\geq 2$ and $d\geq 3$, for the $q$-state ferromagnetic Ising/Potts model on an $n$-vertex complete $d$-ary tree,
	the relaxation time of the Swendsen-Wang dynamics is $\Theta(1)$ and
	its mixing time is $\Theta(\log{n})$
	in the following cases:
	\begin{enumerate}
		\item the boundary condition is arbitrary and $\beta<\max\left\{\beta_u,\frac{1}{2}\ln\left(\frac{\sqrt{d}+1}{\sqrt{d}-1}\right)\right\}$;
		\item the boundary condition is monochromatic and $\beta$ is arbitrary.
	\end{enumerate}
\end{theorem}

Part (i) of this theorem provides optimal mixing and relaxation times bounds for the SW dynamics under arbitrary boundaries throughout the uniqueness region $\beta < \beta_u$.
In fact, $\beta_u < \frac{1}{2}\ln(\frac{\sqrt{d}+1}{\sqrt{d}-1})$ when $q \le 2(\sqrt{d}+1)$ 
and thus our bound extends to the non-uniqueness region for many combinations of $d$ and $q$.
We note that while the value of the uniqueness threshold $\beta_u$ is known, it does not have a closed form (see~\cite{Haggstrom,BGGSVY}). In contrast,
the reconstruction threshold $\beta_r$ is not known for the Potts model~\cite{SlyRecons,MP}, but
one would expect that part (i) holds for all $\beta < \beta_r$; analogous results are known
for the Glauber dynamics for other spin systems where more precise bounds on the reconstruction
threshold have been established~\cite{BKMP,RSVVY,SlyZhang}.

Previously, only a $\poly(n)$ bound was known for the mixing time of the SW dynamics 
for arbitrary boundary conditions~\cite{Ullrich,BKMP}.
This $\poly(n)$ bound holds for every~$\beta$, but the degree of the polynomial bounding the mixing time is quite large (grows with $\beta$); our bound in part (i) is thus a substantial improvement.


In regards to part (ii) of the theorem, we note that our bound holds for all $\beta$, including the whole low-temperature region. 
The only other case where tight bounds for the SW dynamics are known for the full low-temperature regime is on the geometrically simpler complete graph~\cite{GSV,BSmf}.

Previous (direct) analysis of the speed of convergence of the SW dynamics on trees 
focused exclusively on the special case of the \emph{free} boundary condition~\cite{Huber,CF}, where the dynamics is much simpler as the corresponding 
random-cluster model is trivial (reduces to independent bond percolation); 
this was used by Huber~\cite{Huber} to establish
$O(\log n)$ mixing time of the SW dynamics for all $\beta$ for the special case of the free boundary condition.



We comment briefly on our proof methods next; a more detailed exposition of our approach is provided later in this introduction.
The results in~\cite{MSW04,MSW-Potts} use the VM and EM condition to
deduce optimal bounds for the relaxation and mixing times of the Glauber dynamics; specifically, they analyze its spectral gap and log-Sobolev constant. 
Their methods do not extend to the SW dynamics. 
It can be checked, for example, that the log-Sobolev constant for the SW dynamics is $\Theta({n}^{-1})$, and thus the best possible mixing time one could hope to obtain with such an approach would be $O(n\log n)$.
For \Cref{thm:entropy-main},
we utilize instead new tools introduced by Caputo and Parisi~\cite{CP} to establish a (block) factorization of entropy. This factorization allows to get a handle on the
\emph{modified} log-Sobolev constant for the SW dynamics.
For \Cref{thm:intro-main}, the main novelty in our approach is a new spectral interpretation of the VM condition that facilitates a factorization of variance, similar to the factorization of entropy from~\cite{CP}.
Lastly, the lower bound from \Cref{thm:lb:intro} is obtained by adapting the framework of Hayes and Sinclair~\cite{HS} to the SW setting using recent ideas from~\cite{BCPSVstoc}.

Finally, we mention that
part (ii) of \cref{thm:cor} has interesting implications related to the speed of convergence of random-cluster model Markov chains on trees under the wired boundary condition. That is, all the leaves are connected through external or ``artificial'' wirings.
The case of the wired boundary condition is the most studied version of the random-cluster model on trees 
(see, e.g.,~\cite{Haggstrom,J})
since, as mentioned earlier, the model  is trivial under the free boundary. 
The random-cluster model, which is parameterized by $p \in (0,1)$ and $q>0$ and is formally defined in \cref{sec:rc}, is intimately connected to the ferromagnetic $q$-sate Potts model when $q \ge 2$ is an integer and $p = 1 - \exp(-\beta)$. 
In particular, there is a variant of SW dynamics for the random-cluster model (by observing the edge configuration after the second step of the chain).

Another standard Markov chain for the random-cluster model is the heat-bath (edge) dynamics, which is the analog of the Glauber dynamics on spins for random-cluster configurations.  
Our results for the random-cluster dynamics are the following.

\begin{theorem}
	\label{thm:rc:intro}
	For all integer $q\geq 2$, $p \in (0,1)$, and $d\geq 3$, for the random-cluster model on an $n$-vertex
	complete $d$-ary tree with {\em wired boundary condition},
	the mixing time of the Swendsen-Wang dynamics is $O(\log{n})$.
	In addition, the mixing time of the heat-bath edge dynamics for the random-cluster model is $O(n\log{n})$.
\end{theorem}
To prove these results, we use a factorization of entropy in the joint spin-edge space,
as introduced in~\cite{BCPSVstoc};
they cannot be deduced from the mixing time bounds for the Glauber dynamics for the Potts model in~\cite{MSW04,MSW-Potts}. 

Our final result shows that while random-cluster dynamics mix quickly under the wired boundary condition, there are random-cluster boundary conditions that cause an exponential slowdown for both the SW dynamics and the heat-bath edge dynamics for the random-cluster model. 

\begin{theorem}
	\label{thm:rc:lb}
	For all $q\geq 2$, all $d\geq 3$,
	consider the random-cluster model on an $n$-vertex 
	complete $d$-ary tree. Then,
	there exists $p \in (0,1)$ and a random-cluster boundary condition
	such that the mixing times of the Swendsen-Wang dynamics 
	and of the heat-bath edge dynamics is $\exp(\Omega(\sqrt{n}))$.
\end{theorem}
We prove this result extending ideas from~\cite{BGV}.
In particular, we prove a general theorem (see Theorem~\ref{thm:lb:general})
that allows us to transfer slow mixing results for the edge dynamics on other graphs to the tree, for a carefully constructed tree boundary condition and a suitable $p$. Theorem~\ref{thm:rc:lb} then follows from any of the known slow mixing results for the edge dynamics~\cite{GLP,GSVY,Ullrich}.
The proof of Theorem~\ref{thm:lb:general} uses the random-cluster boundary condition to embed an arbitrary graph $G$ on the tree; a set with bad conductance for the chain on $G$ is then lifted to the tree.

\medskip\noindent\textbf{Our techniques.} \ Our first technical contribution 
is a reinterpretation and generalization of the VM condition as a bound on the second eigenvalue of a certain stochastic matrix which we denote by ${P^\uparrow}{P^\downarrow}$.
The matrices ${P^\uparrow}$ and ${P^\downarrow}$ are distributional matrices corresponding to the distribution at a vertex $v$ given the spin configuration of the set $S_v$ of all its descendants at distance at least $\ell$ and vice versa.  
These matrices are inspired by the recent results in~\cite{AL,ALO} utilizing high-dimensional expanders; see \cref{sec:variance} for their precise definitions.

Our new spectral interpretation of the VM condition allows us to factorize it and obtain an equivalent global variant we call \emph{Parallel Variance Mixing (PVM)}.  
%
While the VM condition
signifies 
the exponential decay with distance of the correlations between a vertex $v$ and the set $S_v$ (and is well-suited for the analysis of local Markov chains), 
the PVM condition captures instead the decay rate of set-to-set correlations between the set of all the vertices at a fixed level of the tree and the set of all their descendants at distance at least $\ell$.
The PVM condition facilitates the analysis of a block dynamics with a constant number of blocks each of linear volume. We call this variant of block dynamics the \textit{tiled block dynamics} as each block consists of a maximal number of non-intersecting subtrees of constant size (i.e., a \emph{tiling}); see \cref{fig:tiling}. We use the PVM condition to show that the spectral gap of the tiled block dynamics is $\Omega(1)$, and
a generic comparison between the block dynamics and the SW dynamics yields \cref{thm:intro-main}. 

Our proof of~\Cref{thm:entropy-main} follows a similar strategy.
We first obtain a global variant of the EM condition, analogous to the PVM condition but for entropy. For this, we use a recent result of Caputo and Parisi~\cite{CP}.
From this global variant of the EM condition 
we deduce a factorization of entropy into the even and odd subsets of vertices.
(The parity of a vertex is that of its distance to the leaves of the tree.) 
The even-odd factorization of entropy was recently shown in~\cite{BCPSVstoc} to imply $O(\log n)$ mixing of the SW on general biparte graphs.

\medskip\noindent\textbf{Paper organization.}
The rest of the paper is organized as follows. 
\cref{sec:prelim} contains some standard definitions and facts we use in our proofs. 
In \Cref{sec:variance,sec:ent-proof} we prove~\Cref{thm:intro-main,thm:entropy-main}, respectively. 
Our general comparison result between the SW dynamics and the block dynamics is provided in \cref{sec:sw}. 
Our results for the random-cluster model dynamics (\cref{thm:rc:intro,thm:rc:lb}) are given in \cref{sec:rc,sec:rc:slow}, respectively, and our lower bound for the SW dynammics (\cref{thm:lb:intro}) is
proved in~\Cref{sec:lower-bound}.

\section{Preliminaries}
\label{sec:prelim}

We introduce some notations and facts that are used in the remainder of the paper. 

\medskip\noindent\textbf{The Potts model on the $d$-ary tree.} \ For $d\ge 2$, let $\T^d = (\V,\Edges)$ denote the rooted infinite $d$-ary tree in which every vertex (including the root) has exactly $d$ children. 
We consider the complete finite subtree of $\T^d$ of height $h$, which we denote by $T = T^d_h = (V(T),E(T))$. 
We use $\partial T$ to denote the external boundary of $T$; i.e., the set of
vertices in $\V \setminus V(T)$ incident to the leaves of $T$.
We identify subgraphs of $T$ with their vertex sets. In particular, for $A \subseteq V(T)$
we use $E(A)$ for the edges with both endpoints in $A$, $\partial A$ for the external boundary of $A$ (i.e., the vertices in $(T \cup \partial T) \setminus A$ adjacent to $A$), and, with a slight abuse of notation, we write $A$ also for the induced subgraph $(A,E(A))$.
When clear from context, we simply use $T$ for the vertex set $V(T)$.

A configuration of the Potts model is an assignment of spins $[q] = \{1,\dots,q\}$ to the vertices of the graph.
For a fixed spin configuration $\tau$ on the infinite tree $\T^d$, 
we use $\Omega^\tau = [q]^{T \cup \partial T}$ to denote the set of
configurations of $T$ that agree with $\tau$ on $\partial T$.
Hence, $\tau$ specifies a boundary condition for $T$.
More generally, for any $A \subseteq T$ and any $\eta \in \Omega^\tau$, let $\Omega_A^\eta \subseteq \Omega^\tau$ denote the set of configurations of $T$ that agree with $\eta$ on $(T \cup \partial T) \setminus A$.
We use $\mu_A^\eta$ to denote the Gibbs distribution over $\Omega_A^\eta$,
so for $\sigma \in \Omega_A^\eta$ we have
$$
\mu_{A}^\eta (\sigma) := \frac{1}{Z} \exp\Big(-\beta  \sum\nolimits_{\{u,v\} \in E(A \cup \partial A)} \1(\sigma_u \neq \sigma_v)  \Big),
$$
where $Z$ is a normalizing constant (or partition function). 
For $\sigma \notin \Omega_A^\eta$, we set $\mu_{A}^\eta (\sigma) = 0$. 


\medskip\noindent\textbf{The tiled block dynamics.} \ Let $\mathcal U = \{U_1, . . . , U_r\}$ be a collection of subsets (or blocks) such that $T = \bigcup_i U_i$. The (heat-bath) block dynamics with blocks $\mathcal U$
is a standard Markov chain for the Gibbs distribution $\mu_T^\tau$.
If the configuration at time $t$ is $\sigma_t$, the next configuration $\sigma_{t+1}$ is generated as follows:

\vspace{-1mm}
\begin{enumerate}
	\item Pick an integer $j \in \{1,2,\dots,r\}$ uniformly at random;
	\item Draw a sample $\sigma_{t+1}$ from the conditional Gibbs distribution $\mu_{U_j}^{\sigma_t}$; that is,  
	update the configuration in $U_j$ with a new configuration distributed according to the conditional measure in $U_j$ given the configuration of $\sigma_t$ on $(T\cup\partial T) \setminus U_j$ and the boundary condition~$\tau$.
\end{enumerate}


\noindent
We consider a special choice of blocks, where each block is a disjoint union of small subtrees of constant height forming a tiling structure. 
For $0\le i\le h+1$, let $L_i$ denote the set of vertices of $T$ that are of distance exactly $i$ from the boundary $\partial T$; in particular, $L_0 = \emptyset$ and $L_{h+1}$ contains only the root of $T$. 
(It will be helpful to define $L_i = \emptyset$ for $i<0$ or $i>h+1$.)
Let $F_i = \cup_{j \le i} L_j$ be the set of vertices at distance at most $i$ from $\partial T$; 
then $F_0 = \emptyset$ and $F_{h+1} = T$.
We further define $F_i = \emptyset$ for $i<0$ and $F_i = T$ for $i>h+1$. 
For each $i \in \N^+$ let
\begin{equation}
\label{eq:bi}
B_i^\ell = F_i \bs F_{i-\ell} = \bigcup_{i-\ell < j \le i} L_j.
\end{equation}
In words, $B_i^\ell$ is the collection of all the subtrees of $T$ of height $\ell-1$ 
with roots at distance exactly $i$ from $\partial T$; see Figure~\ref{fig:tiling}(b). 
Finally, for each $1\le j \le \ell+1$, we define 
\begin{equation}
\label{eq:tj}
T_j^\ell = \bigcup_{0 \le k \le \frac{h+\ell-j}{\ell+1}} B_{j+k(\ell+1)}^\ell.
\end{equation}
The set $T_j^\ell$ contains all the subtrees of $T$ 
whose roots are at distance $j + k(\ell+1)$ from $\partial T$ for some non-negative integer $k$; the height of each subtree (except the top and bottom ones) is $\ell-1$.
Also notice that all the subtrees in $T_j^\ell$ are at (graph) distance at least $2$ from each other, and thus they create a tiling pattern over $T$. Therefore, we call the block dynamics with blocks $\mathcal U = \{T_1^\ell,\dots,T_{\ell+1}^\ell\}$ the \emph{tiled block dynamics}; see Figure \ref{fig:tiling}(c). 
The transition matrix of the tiled block dynamics is denoted by $\Ptb$. 


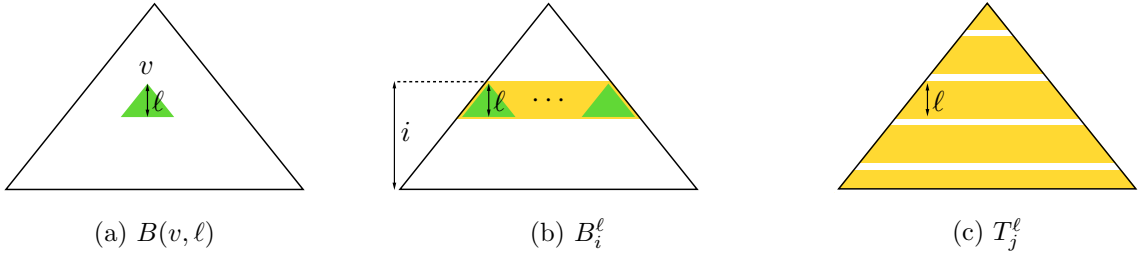
\begin{figure}[t]
	\centering
	\captionsetup{justification=centering}
	\begin{subfigure}{0.32\textwidth}
		\centering
		\begin{tikzpicture}
		\node at (0,0) {\includegraphics[scale=0.14]{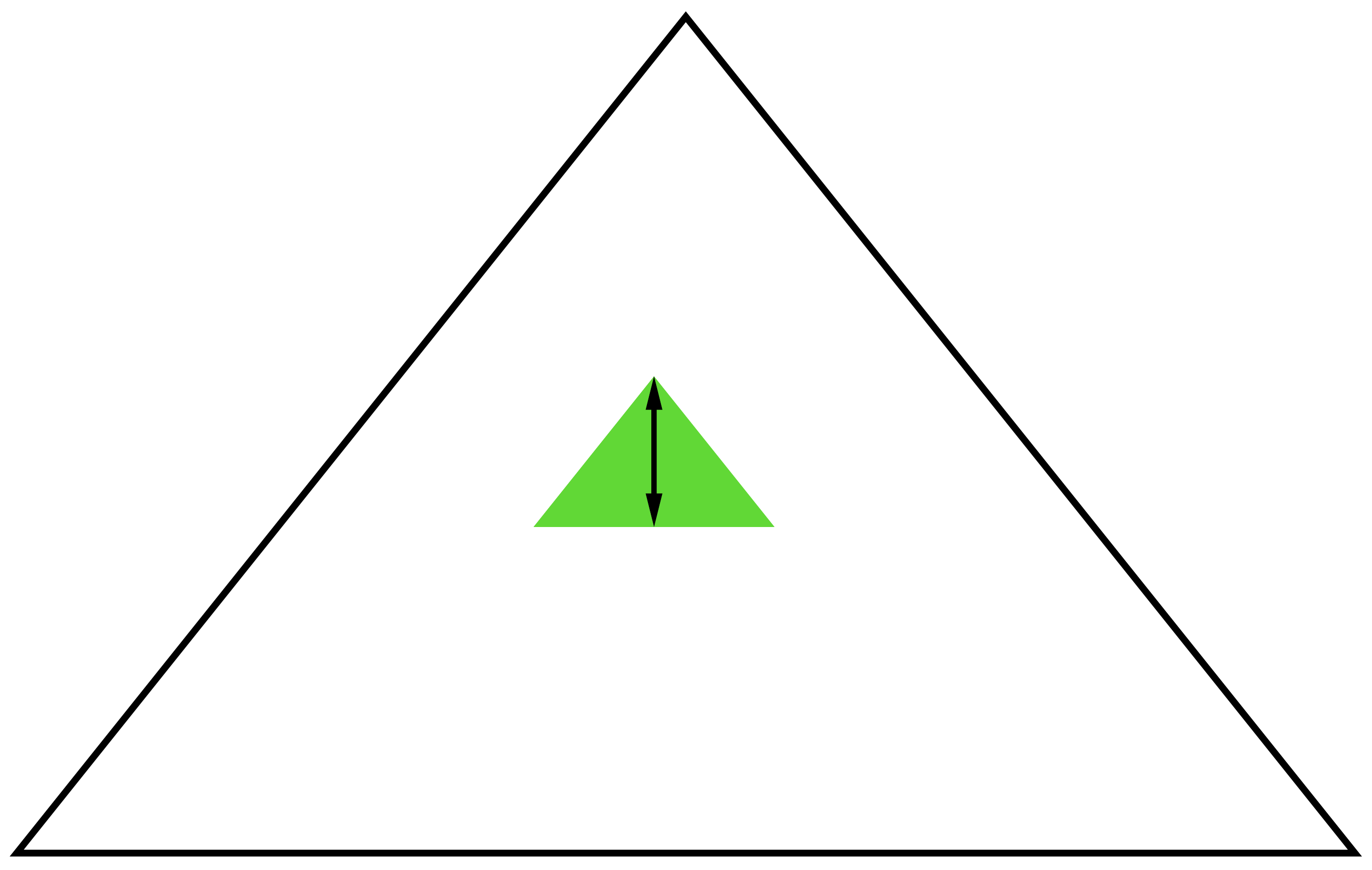}};
		\node at (0.05,-0.06) {$\ell$};
		\node at (-0.11,0.38) {$v$};
		\end{tikzpicture}
		\caption{$B(v,\ell)$}
	\end{subfigure}
	\begin{subfigure}{0.33\textwidth}
		\centering
		\begin{tikzpicture}
		\node at (0,0) {\includegraphics[scale=0.14]{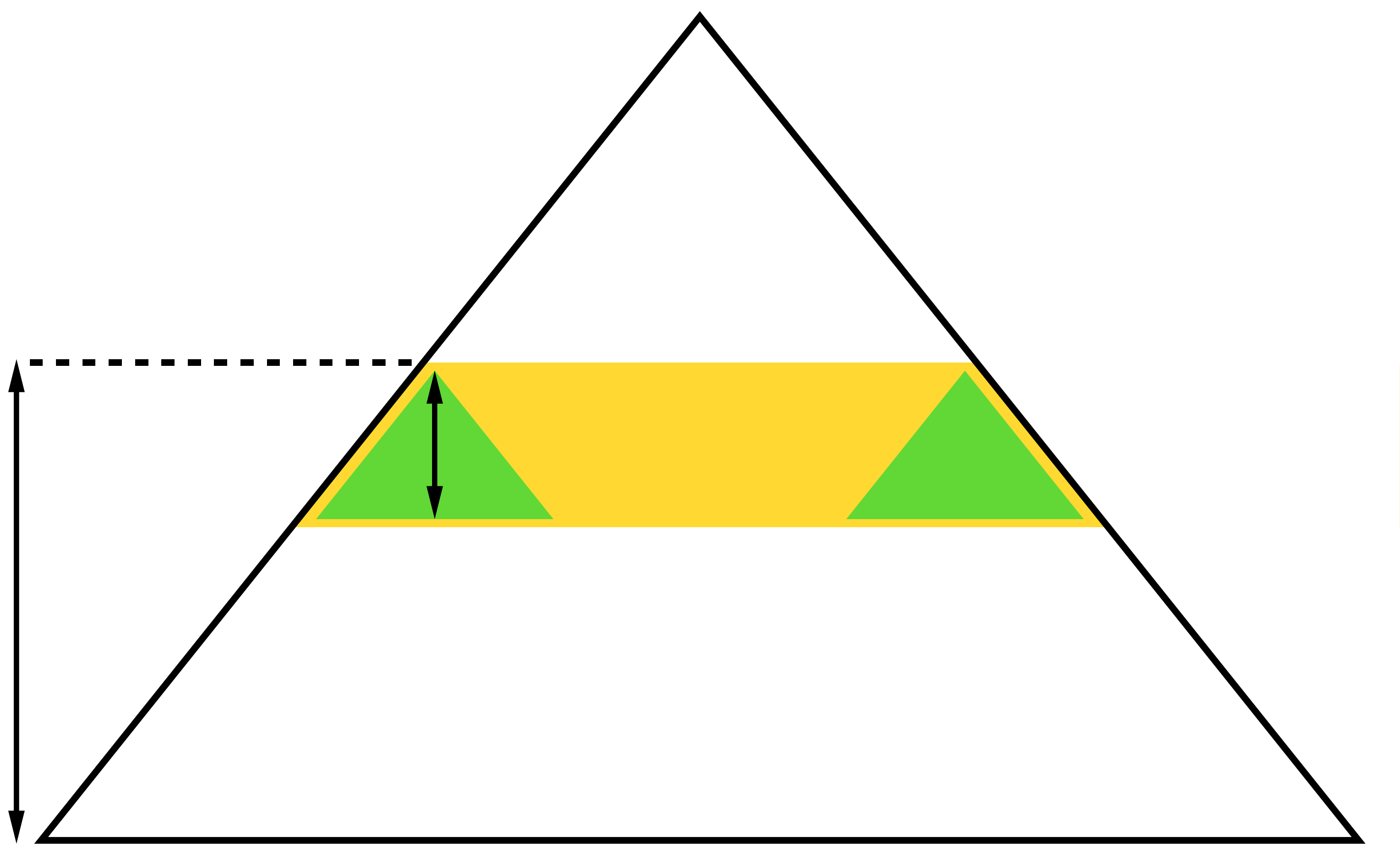}};
		\node at (-0.65,-0.06) {$\ell$};
		\node at (0.03,-0.06) {$\cdots$};
		\node at (-1.9,-0.47) {$i$};
		\node at (2.61,-0.59) {};
		\end{tikzpicture}
		\caption{$B_i^\ell$}
	\end{subfigure}
	\begin{subfigure}{0.33\textwidth}
		\centering
		\begin{tikzpicture}
		\node at (0,0) {\includegraphics[scale=0.14]{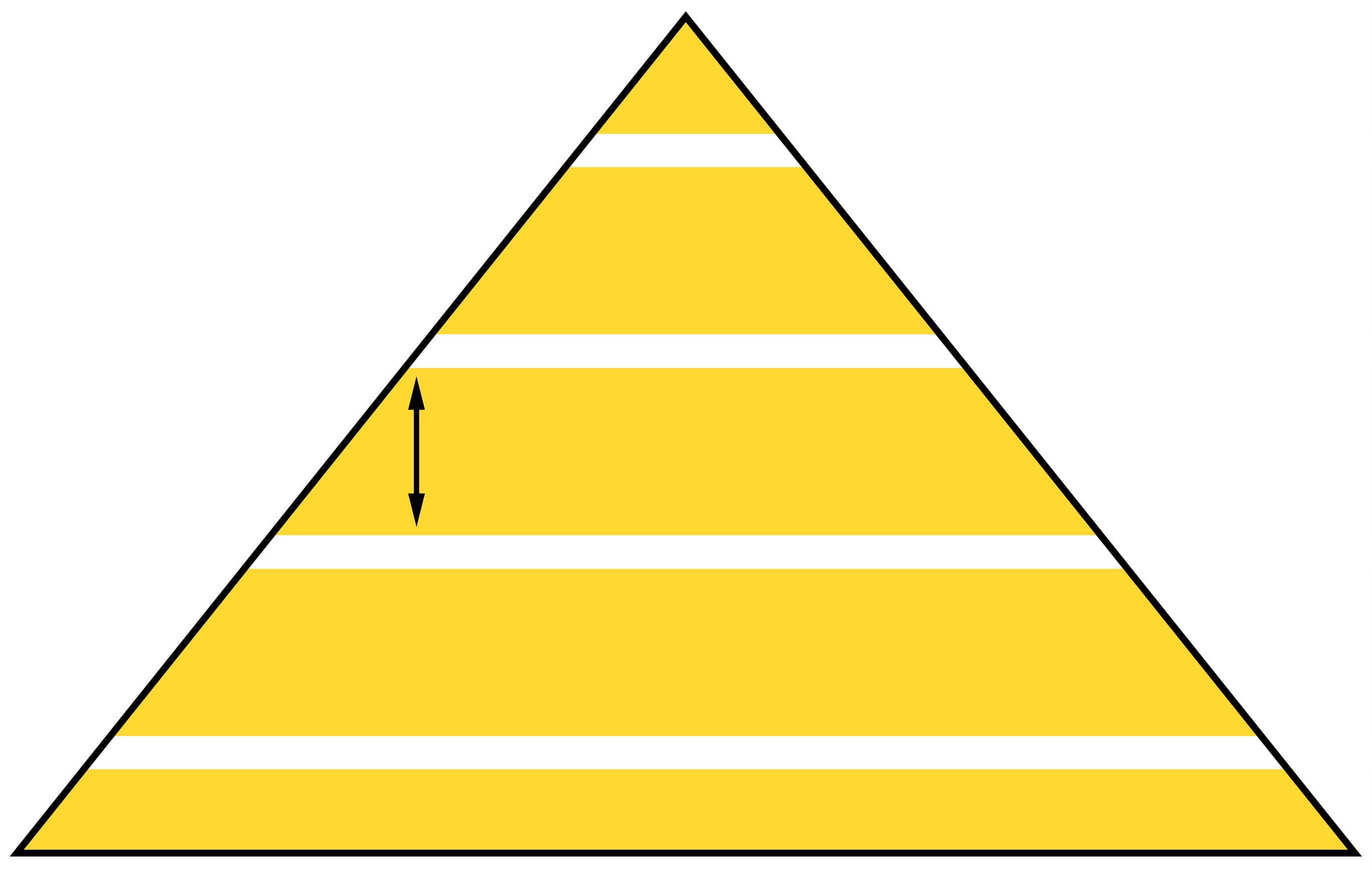}};
		\node at (-0.65,-0.06) {$\ell$};
		\end{tikzpicture}
		\caption{$T_j^\ell$}
	\end{subfigure}
	\caption{An illustration of the sets $B(v,\ell)$, $B_i^\ell$, and $T_j^\ell$, where $\ell$ represents the number of levels.}
	\label{fig:tiling}	
\end{figure}


\medskip\noindent\textbf{Mixing and relaxation times.} \ Let $P$ be the transition matrix of an ergodic Markov chain over a finite set $\Phi$ with stationary distribution $\nu$.
We use $P^t(X_0,\cdot)$ to denote the distribution of the chain after $t$ steps starting from $X_0 \in \Phi$. The mixing time of $P$ is defined as
$
\tmix(P) = \max\limits_{X_0 \in \Phi}\min \left\{ t \ge 0 : {\|P^t(X_0,\cdot)-\nu\|}_{\textsc{tv}} \le 1/4 \right\},
$
where $\|\cdot\|_{\textsc{tv}}$ denotes total variation distance.

When $P$ is reversible, its spectrum is real and we let $1 = \lambda_1 > \lambda_2 \ge ... \ge \lambda_{|\Phi|} \geq -1$ denote its eigenvalues ($1 > \lambda_2$ when $P$ is irreducible).
The {\it absolute spectral gap} of $P$ is defined by $\GAP(P) = 1 - \lambda^*$, where $\lambda^* = \max\{|\lambda_2|,|\lambda_{|\Phi|}|\}$. If $P$ is ergodic (i.e., irreducible and aperiodic), then $\GAP(P)>0$, and
it is a standard fact that if $\nu_{\mathrm{min}} = \min_{x \in \Phi} \nu(x)$, then
\begin{equation}
\label{prelim:eq:gap-lower}
\left(\GAP(P)^{-1}-1\right) \log 2 \le \tmix(P) \le \GAP(P)^{-1} \log \left({4}{\nu_{\mathrm{min}}^{-1}}\right);
\end{equation}
see~\cite{LPW}.
The \textit{relaxation time} of the chain is defined as $\GAP(P)^{-1}$.

\medskip\noindent\textbf{Analytic tools.} \ We review next some useful tools from functional analysis; we refer the reader to~\cite{MT,SC} for more extensive background.
We can endow $\R^{\Phi}$ with the inner product 
$\langle f,g \rangle_\nu = \sum_{x \in \Phi} f(x)g(x)\nu(x)$ for two functions $f,g:\Phi\to\R$.
The resulting Hilbert space is denoted by $L_2(\nu) = (\R^{\Phi},\langle \cdot,\cdot \rangle_\nu)$ and $P$ defines an operator from $L_2(\nu)$ to $L_2(\nu)$.

Let $\allone: \Phi \to \R$ be the constant ``all $1$'' function (i.e., $\allone(x) = 1$ $\forall x\in \Phi$) 
and let $I$ denote the identity mapping over all functions (i.e., $If = f$ for all $f:\Phi \to \R$). 
We then define:
\begin{align*}
\E_\nu(f) & = \sum_{x \in \Phi} f(x)\nu(x) = \ip{f}{\allone}_\nu,~\text{and} \\ 
\Var_\nu(f) &= \E_\nu(f^2) - \E_\nu(f)^2 = \ip{f}{(I-\allone\nu) f}_\nu 
\end{align*}
as the expectation and variance of the function $f$ with respect to (w.r.t.) the measure $\nu$. Likewise, for a function $f: \Omega \to \R_{\ge 0}$ we define the entropy of $f$ with respect to $\nu$ as
$
\Ent_\nu(f) = \E_\nu\big[f \log \big(\frac{f}{\E_\nu(f)}\big)\big].
$

Often, we will consider $\nu$ to be the conditional Gibbs distribution $\mu_A^\eta$ for some $A \subseteq T$ and $\eta\in \Omega$.
In those cases, to simplify the notation, we shall write 
$\E_A^\eta(f)$ for $\E_{\mu_A^\eta}(f)$,
$\Var_A^\eta(f)$ for $\Var_{\mu_A^\eta}(f)$,
and $\Ent_A^\eta(f)$ for $\Ent_{\mu_A^\eta}(f)$.

The \textit{Dirichlet form} of a reversible Markov chain with transition matrix $P$ is defined as
\begin{equation}
\label{eq:df}
\mathcal{E}_P(f,f) = \inner{f}{(I-P)f}{\nu} 
= \frac{1}{2} \sum_{x,y \in \Phi} \nu(x) P(x,y) (f(x) - f(y))^2,
\end{equation}
for any $f:\Phi \to \R$.
We say $P$ is \textit{positive semidefinite} if $\langle  f,P f \rangle_\nu \ge 0$ for all functions $f:\Phi \to \R$. 
In this case $P$ has only nonnegative eigenvalues.
If $P$ is positive semidefinite, then
the absolute spectral gap of $P$ satisfies
\begin{equation}
\label{eq:sw:gap}
\GAP(P) = 1 - \lambda_2 = 
\inf_{\substack{f: \Phi \to \R\\ \Var_\nu(f) \neq 0}} \frac{\mathcal{E}_{P} (f,f)}{\Var_\nu (f)}.
\end{equation}

\section{Variance Mixing implies fast mixing: Proof of \texorpdfstring{\cref{thm:intro-main}}{Theorem 1} }
\label{sec:variance}

We start with the formal definition of the \emph{Variance Mixing (VM)} condition introduced by
Martinelli, Sinclair and Weitz~\cite{MSW04}.
Throughout this section, we consider the Potts model
on the $n$-vertex $d$-ary complete tree $T=T^d_h$ 
with a fixed boundary condition $\tau$; hence, for ease of notation we set
$\mu := \mu_T^\tau$ and $\Omega := \Omega^\tau$.

For $v\in T$, let $T_v$ denote the subtree of $T$ rooted at $v$. 
For boundary condition $\eta \in \Omega$ and a function $g: \Omega_{T_v}^\eta \to \R$, we define the function $g_v: [q] \to \R$ as the conditional expectation 
\begin{equation}
\label{eq:g}
g_v(a) = \E_{T_v}^\eta[g \mid \sigma_v = a]
= \sum_{\sigma \in \Omega_{T_v}^\eta: \sigma_v = a} \mu_{T_v}^\eta(\sigma \mid \sigma_v = a) g(\sigma). 
\end{equation}
In words, $g_v(a)$ is the conditional expectation of the function $g$ under the distribution $\mu_{T_v}^\eta$ given that the 
root of $T_v$ (i.e, the vertex $v$) is set to spin $a \in [q]$.
We also consider the expectation and variance of $g_v$
w.r.t.\ the projection of $\mu_{T_v}^\eta$ on $v$. In particular,
\begin{align*}
\E_{T_v}^\eta[g_v] &= 
\sum\nolimits_{a\in [q]} \mu_{T_v}^\eta(\sigma_v = a) g_v(a)
= \E_{T_v}^\eta[g],
~\text{and} \\
\Var_{T_v}^\eta[g_v] &= \E_{T_v}^\eta[g_v^2] - \E_{T_v}^\eta[g_v]^2.
\end{align*}

\noindent
For an integer $\ell \ge 1$, we define $B(v,\ell)$ as the set of vertices of $T_v$ that are at distance less than $\ell$ from $v$; see Figure~\ref{fig:tiling}(a).
We say that the function $g: \Omega_{T_v}^\eta \to \R$ is independent of the configuration on $B(v,\ell)$ if for all $\sigma, \sigma' \in \Omega_{T_v}^\eta$ such that 
$\sigma(B(v,\ell)) \neq \sigma'(B(v,\ell))$
and $\sigma(T_v \setminus B(v,\ell)) = \sigma'(T_v \setminus  B(v,\ell))$, we have $g(\sigma) = g(\sigma')$. We can now define~$\VM$.

\begin{definition}[Variance Mixing (VM)]\label{def:vm}
	The Gibbs distribution $\mu = \mu_T^\tau$ satisfies $\VM(\ell,\eps)$ if for every $v\in T$, every $\eta \in \Omega$, and every function $g: \Omega_{T_v}^\eta \to \R$ that is independent of the configuration on $B(v,\ell)$, we have
	$
	\Var_{T_v}^\eta(g_v) \le \eps \cdot \Var_{T_v}^\eta(g).
	$
	We say that the VM condition holds if there exist constants $\ell$ and $\eps = \eps(\ell)$ such that $\VM(\ell,\eps)$ holds.
\end{definition}

The VM condition is a spatial mixing property that captures
the rate of decay of correlations, given by $\varepsilon = \varepsilon(\ell)$, with the distance
$\ell$ between $v \in T$ and the set $T_v\setminus B(v,\ell)$.
To see this, note that, roughly speaking, $\Var_{T_v}^\eta(g_v)$ is small when 
$g_v(a) = \E_{T_v}^\eta[g \mid \sigma_v = a]$ is close to $g_v(b) = \E_{T_v}^\eta[g \mid \sigma_v = b]$ for every $a \neq b$. Since $g$ is independent of the configuration on $B(v,\ell)$, 
this can only happen if the spin at $v$, which is at distance $\ell$ 
from $T_v \setminus B(v,\ell)$, has only a small influence on the projections
of the conditional measures $\mu_{T_v}^\eta(\cdot \mid \sigma_v = a)$, $\mu_{T_v}^\eta(\cdot \mid \sigma_v = b)$ to $T_v \setminus B(v,\ell)$.

It was established in~\cite{MSW04,MSW-Potts} that VM implies optimal mixing of the Glauber dynamics;  
this was done by analyzing a block dynamics that updates one random block $B(v,\ell)$ in each step.
This block dynamics behaves similarly to the Glauber dynamics since all blocks are of constant size, and there are a linear number of them; see ~\cite{MSW04,MSW-Potts} for further details. 
Our goal here is to establish optimal mixing of global Markov chains, and thus we require a different spatial mixing condition that captures decay of correlations in a more global manner. For this, we introduce the notion of \emph{Parallel Variance Mixing (PVM)}.
Recall that for $0\le i\le h+1$, $L_i$ is the set all vertices at distance exactly $i$ from the boundary $\partial T$, $F_i = \cup_{j \le i} L_j$, and $B_i^\ell = F_i \bs F_{i-\ell}$; see Figures~\ref{fig:tiling}(b) and~\ref{fig:tiling}(c).

For $1\le i\le h+1$, $\eta \in \Omega$ and $g: \Omega_{F_i}^\eta \to \R$, consider the function 
$g_{L_i}: [q]^{L_i} \to \R$ given by 
$$
g_{L_i}(\xi) = \E_{F_i}^\eta[g \mid \sigma_{L_i} = \xi] = \sum_{\sigma \in \Omega_{F_i}^\eta: \sigma_{L_i} = \xi} \mu_{F_i}^\eta(\sigma \mid \sigma_{L_i} = \xi) g(\sigma),
$$
for $\xi \in [q]^{L_i}$. 
That is, $g_{L_i}(\xi)$ is the conditional expectation of function $g$ under the distribution $\mu_{T_v}^\eta$ conditioned on the configuration of the level $L_i$ being $\xi$.
Thus, we may consider the expectation and variance of $g_{L_i}$
w.r.t.\ the projection of $\mu_{T_v}^\eta$ to $L_i$; namely,
$\E_{F_i}^\eta[g_{L_i}] = \E_{F_i}^\eta[g]$ 
and
$\Var_{F_i}^\eta[g_{L_i}] = \E_{F_i}^\eta[g_{L_i}^2] - \E_{F_i}^\eta[g_{L_i}]^2$.
The PVM condition is defined as follows.

\begin{definition}[Parallel Variance Mixing (PVM)]
	\label{def:pvm}
	The Gibbs distribution $\mu = \mu_T^\tau$ satisfies $\PVM(\ell,\eps)$ if for every $1\le i\le h+1$, every $\eta \in \Omega$, and every function $g: \Omega_{F_i}^\eta \to \R$ that is independent of the configuration on $B_i^\ell$, we have
	$
	\Var_{F_i}^\eta(g_{L_i}) \le \eps \cdot \Var_{F_i}^\eta(g).
	$
	The PVM condition holds if there exist constants $\ell$ and $\eps = \eps(\ell)$ such that $\PVM(\ell,\eps)$ holds.
\end{definition}
PVM is a natural global variant of VM since
$F_i = \bigcup_{v\in L_i} T_v$
and
$B_i^\ell = \bigcup_{v\in L_i} B(v,\ell)$.
We can show that the two properties are actually equivalent.

\begin{theorem}\label{thm:VM=PVM}
	For every $\ell \in \N^+$ and $\eps \in (0,1)$, the Gibbs distribution $\mu$ satisfies $\VM(\ell,\eps)$ if and only if $\mu$ satisfies $\PVM(\ell,\eps)$.
\end{theorem}
In order to show the equivalence between VM and PVM, we introduce a
more general spatial mixing condition which we call \emph{General Variance Mixing (GVM)}.
We define GVM for general product distributions (see Definition~\ref{def1})
and reinterpret VM and PVM as special cases of this condition.
This alternative view of VM and PVM in terms of GVM is quite useful since we can recast the GVM condition as a bound on the spectral gap of a certain Markov chain; this is one key insight in the proof of Theorem~\ref{thm:rc:intro} and is discussed in detail in Section~\ref{sec:VM=PVM}.

Now, while VM implies optimal mixing of the Glauber dynamics, we can show that 
PVM implies a constant bound on the spectral gap of the tiled block dynamics. 
Recall that this is the heat-bath block dynamics with block collection $\mathcal U = \{T_1^\ell,\dots,T_{\ell+1}^\ell\}$ defined in \cref{sec:prelim}.

\begin{theorem}\label{thm:PVM-gap}
	If there exist $\ell\in\N^+$ and $\delta\in(0,1)$ such that $\mu = \mu_T^\tau$ satisfies $\PVM(\ell,\eps)$ for $\eps = \frac{1-\delta}{2(\ell+1)}$, then the relaxation time of the tiled block dynamics is at most $2(\ell+1)/\delta$. 
\end{theorem}

To prove~\cref{thm:PVM-gap}, we adapt the methods from~\cite{MSW04,MSW-Potts} to our global setting; see~\cref{sec:PVM-gap}. Our result for the spectral gap of the SW dynamics (\cref{thm:intro-main}) is then obtained through comparison with the tiled block dynamics.
We prove the following comparison result between the SW dynamics and a large class of block dynamics, which could be of independent interest.

\begin{theorem}
	\label{thm:sw-block}
	Let $\mathcal D = \{D_1,\dots,D_m\}$ be such that $D_i \subseteq T$ and $\cup_{i=1}^m D_i = T$. 
	Suppose that each block $D_k$ is such that $D_k = \cup_{j=1}^{\ell_k} D_{kj}$ where $\dist(D_{kj}, D_{kj'}) \ge 2$ for every $j \neq j'$ and let $\mathrm{vol}(\mathcal D) = \max_{k,j} |D_{kj}|$. 
	Let $\mathcal B_\mathcal D$ be the transition matrix of the (heat-bath) block dynamics with blocks $\mathcal D$ and let	
	$\PSW$ denote the transition matrix for the SW dynamics.
	Then,
	$
	\GAP(\PSW)  \ge \exp(-O(\mathrm{vol}(\mathcal D))) \cdot \GAP(\mathcal B_\mathcal D).
	$
\end{theorem}

The blocks of the tiled block dynamics satisfy all the conditions in this theorem, and, in addition, $\mathrm{vol}(\mathcal D) = O(1)$. Hence, combining all the results stated in this section, we see that \cref{thm:intro-main} from introduction follows.

\begin{proof}[Proof of \cref{thm:intro-main}]
	Follows from \cref{thm:VM=PVM,thm:PVM-gap,thm:sw-block}.
\end{proof}


\subsection{Equivalence between VM and PVM: Proof of \texorpdfstring{\cref{thm:VM=PVM}}{Theorem 7}}
\label{sec:VM=PVM}

In this section we establish the equivalence between VM and PVM. 
We start with the definition of \emph{General Variance Mixing (GVM)}.
Let $\Phi$ and $\Psi$ be two finite sets and let $\rho(\cdot,\cdot)$ be an arbitrary joint distribution supported on $\Phi \times \Psi$. 
Denote by $\nu$ and $\pi$ the marginal distributions of $\rho$ over $\Phi$ and $\Psi$, respectively. That is,
for $x \in \Phi$ we have $\nu(x) = \sum_{y\in \Psi} \rho(x,y)$, and for 
$y\in \Psi$ we have $\pi(y) = \sum_{x\in \Phi} \rho(x,y).$
We consider two natural matrices associated to $\rho$. 
For $x\in \Phi$ and $y\in \Psi$, define 
\begin{equation}
\label{eq:matrices}
P^\uparrow(x,y) = \rho(y \mid x) = \frac{\rho(x,y)}{\nu(x)},~\text{and }\quad\quad P^\downarrow(y,x) = \rho(x\mid y) = \frac{\rho(x,y)}{\pi(y)};
\end{equation}
$P^\uparrow$ is a $|\Phi| \times |\Psi|$ matrix while $P^\downarrow$ is a $|\Psi| \times |\Phi|$ matrix. 
In addition, observe that $P^\uparrow P^\downarrow$ and $P^\downarrow P^\uparrow$ are transition matrices of Markov chains reversible w.r.t.\ $\nu$ and $\pi$, respectively.

\begin{definition}[GVM for $\rho$]\label{def1}
	We say that the joint distribution $\rho$ satisfies $\GVM(\eps)$ if for every function $f: \Phi \to \R$ we have
	$
	\Var_\pi(P^\downarrow f) \le \eps \cdot \Var_\nu(f).
	$
\end{definition}

One key observation in our proof is that the GVM condition can be expressed in term of the spectral gaps of the matrices $P^\uparrow P^\downarrow$ and $P^\downarrow P^\uparrow$.

\begin{lemma}\label{le1}
	The joint distribution $\rho$ satisfies $\GVM(\eps)$ 
	if and only if 
	$ \GAP(P^\uparrow P^\downarrow) = \GAP(P^\downarrow P^\uparrow) \ge 1-\eps. $ 
\end{lemma}


Before providing the proof of \cref{le1}, we recall the definition of the \emph{adjoint} operator. 
Let $S_1$ and $S_2$ be two Hilbert spaces with inner products $\langle\cdot,\cdot\rangle_{S_1}$ and $\langle\cdot,\cdot\rangle_{S_2}$ respectively, and let $K:S_2 \rightarrow S_1$ be a bounded linear operator. The {adjoint} of $K$ is the unique operator $K^*:S_1 \rightarrow S_2$ satisfying $\langle f,Kg \rangle_{S_1} = \langle K^*f,g \rangle_{S_2}$ for all $f \in S_1$ and $g \in S_2$. When $S_1 = S_2$, $K$ is called \textit{self-adjoint} if $K = K^*$.
We can now provide the proof of \cref{le1}.

\begin{proof}[Proof of \cref{le1}]
	It is straightforward to check that $P^\uparrow \allone = \allone$, $P^\downarrow \allone = \allone$, $\nu P^\uparrow = \pi$, $\pi P^\downarrow = \nu$, 
	and that the operator $P^\uparrow: L_2(\pi) \to L_2(\nu)$ is the adjoint of the operator $P^\downarrow: L_2(\nu) \to L_2(\pi)$. 
	Hence, both $P^\uparrow P^\downarrow$ and $P^\downarrow P^\uparrow$ are positive semidefinite and have the same multiset of non-zero eigenvalues. 
	Now,  for $f:\Phi \to \R$, we have 
	\[
	\Var_\pi(P^\downarrow f) 
	= \ip{P^\downarrow f}{(I-\allone \pi) P^\downarrow f}_\pi
	= \ip{f}{P^\uparrow (I-\allone \pi) P^\downarrow f}_\nu
	= \ip{f}{P^\uparrow P^\downarrow f}_\nu - \ip{f}{\allone \nu f}_\nu. 
	\] 
	Therefore, 	$\Var_\pi(P^\downarrow f) \le \eps \cdot \Var_\nu(f)$ holds if and only if
	\begin{align*}
	\ip{f}{P^\uparrow P^\downarrow f}_\nu - \ip{f}{\allone \nu f}_\nu &\le \eps \cdot \left( \ip{f}{f}_\nu - \ip{f}{\allone \nu f}_\nu \right)\\
	\Leftrightarrow \quad  \ip{f}{(I - P^\uparrow P^\downarrow) f}_\nu &\ge (1-\eps) \cdot \ip{f}{(I - \allone \nu) f}_\nu\\
	\Leftrightarrow \quad  \EE_{P^\uparrow P^\downarrow}(f,f) &\ge (1-\eps) \cdot \Var_\nu(f). 
	\end{align*}
	The lemma then follows from \eqref{eq:sw:gap}. 
\end{proof}

We provide next the proof of \cref{thm:VM=PVM}, which follows from \Cref{le1}
and interpretations of $\VM$ and $\PVM$ by $\GVM$. 
Given $F=A\cup B \subseteq T$ and $\eta \in \Omega$, let $P^\uparrow = (P_F^\eta)_{A\uparrow B}$ denote
the $q^{|A \setminus B|} \times q^{|B \setminus A|}$ stochastic matrix indexed by the configurations on the sets $A \setminus B$ and $B \setminus A$, 
such that for $\xi \in [q]^{A \setminus B}$ and $\xi' \in [q]^{B\setminus A}$ we have
$
P^\uparrow(\xi,\xi') = \mu_F^\eta(\sigma_{B\setminus A} = \xi' \mid \sigma_{A\setminus B} = \xi). 
$
In words, $P^\uparrow$ corresponds to the transition matrix that given the configuration $\xi$ in $A \setminus B$ updates the configuration in $B \setminus A$ from the conditional distribution $\mu_F^\eta(\cdot \mid \xi)$.  
We define in a similar manner the $q^{|B\setminus A|} \times q^{|A\setminus B|}$ stochastic matrix $P^\downarrow = (P_F^\eta)_{B\downarrow A}$ where for $\xi' \in [q]^{B\setminus A}$ and $\xi \in [q]^{A \setminus B}$ we have
$
P^\downarrow(\xi',\xi) = \mu_F^\eta(\sigma_{A\setminus B} = \xi \mid \sigma_{B\setminus A} = \xi'). 
$

If we set $\rho$ to be the marginal of $\mu_F^\eta$ on $(A \setminus B) \cup (B\setminus A)$, then $\Phi = [q]^{A \setminus B}$, $\Psi = [q]^{B\setminus A}$,
and $\nu$ and $\pi$ are the marginals of $\mu_F^\eta$ on $A \setminus B$ and $B\setminus A$, respectively. Therefore, according to~\cref{def1}, $\GVM(\varepsilon)$ holds for the marginal of $\mu_F^\eta$ on $(A \setminus B) \cup (B\setminus A)$ if 
$
\Var_\pi(P^\downarrow f) \le \eps \cdot \Var_\nu(f)
$
for every function $f: \Phi \to \R$.

Now, note that a function $g: \Omega_F^\eta \to \R$ independent of $B$ only depends on the configuration on $A \setminus B$. Thus, for fixed $\eta$, $g$ induces a function $f: \Phi \to \R$; in particular, $\Var_F^\eta(g) = \Var_\nu(f)$.
Moreover, letting $g_{B\setminus A}(\xi) := \E_F^\eta[g \mid \sigma_{B \setminus A} = \xi]$, we have $g_{B\setminus A}(\xi)= P^\downarrow f (\xi)$ for every $\xi \in \Psi =  [q]^{B\setminus A}$, and so $\Var_F^\eta(g_{B\setminus A})= \Var_\pi(P^\downarrow f)$.
Consequently, we arrive at the following equivalences between $\VM$, $\PVM$ and $\GVM$.

\begin{proposition} \label{prop:eq} \
	\begin{enumerate}
	\item The Gibbs distribution $\mu$ satisfies $\VM(\ell,\epsilon)$ if and only if
	for every $v \in T$ and $\eta \in \Omega$,
	$\GVM(\varepsilon)$ holds for the marginal of $\mu_{T_v}^\eta$ on $(T_v \setminus B(v,\ell)) \cup \{v\}$.
	\item The Gibbs distribution $\mu$ satisfies $\PVM(\ell,\epsilon)$ if and only if
	for every $i$ such that $1\le i \le h+1$ and $\eta \in \Omega$,
	$\GVM(\varepsilon)$ holds for the marginal of $\mu_{F_i}^\eta$ on $(F_i \setminus \cup_{v \in L_i} B(v,\ell)) \cup L_i$.
	\end{enumerate}
\end{proposition}	

\noindent
To see part 1 simply note that in the notation above, we can set $F = T_v$, $A = T_v \setminus v$ and $B = B(v,\ell)$. For part 2, we set $F = F_i$, $A = F_{i-1}$ and $B = B_i^\ell$.

\begin{proof}[Proof of \cref{thm:VM=PVM}]
	From~\cref{prop:eq,le1}, $\VM(\ell,\epsilon)$ holds if and only if
	$\GAP(Q_v) \ge 1-\varepsilon$ for every $v \in T$ and $\eta \in \Omega$, where $Q_v = (P_{T_v}^\eta)_{B(v,\ell)\downarrow (T_v \setminus v)} (P_{T_v}^\eta)_{(T_v \setminus v)\uparrow B(v,\ell)}$.
	Similarly, $\mu$ satisfies $\PVM(\ell,\epsilon)$ if and only if
	$\GAP(Q_{L_i}) \ge 1-\varepsilon$ for every $i$ such that $1\le i \le h+1$ and $\eta \in \Omega$, where $Q_{L_i} = (P_{F_i}^\eta)_{{B_i^\ell}\downarrow {F_{i-1}}} (P_{{F_i}}^\eta)_{{F_{i-1}}\uparrow {B_i^\ell}}$.
	
	Since $F_i = \bigcup_{v\in L_i} T_v$ and the $T_v$'s are at distance at least two from each other,  $\mu_{F_i}^\eta(\sigma_{L_i} = \cdot)$ is a product distribution; in particular $\mu_{F_i}^\eta(\sigma_{L_i} = \cdot) = \prod_{v \in L_i} \mu_{T_v}^\eta(\sigma_v = \cdot)$ and the chain with transition matrix $Q_{L_i}$ is a product Markov chain where each component corresponds to $Q_v$ for some $v \in L_i$. 
	A standard fact about product Markov chains, see, e.g.,~\cite[Lemma 4.7]{BCSV}, then implies that
	$
	\GAP(Q_{L_i}) = \min_{v \in L_i} \GAP(Q_v) 
	$
	and the result follows.
\end{proof}

\subsection{PVM implies fast mixing of the tiled block dynamics: Proof of~\cref{thm:PVM-gap}}
\label{sec:PVM-gap}



In this section we prove \cref{thm:PVM-gap} by showing $\GAP(\Ptb) = \Theta(1)$ when the PVM condition holds; recall that $\Ptb$ denotes the transition matrix of the tiled block dynamics defined in~\cref{sec:prelim}. 
We introduce some useful simplification of our notation next.
For $A \subseteq T$ and $f: \Omega \to \R$,
we define functions $\Ex_A(f): \Omega \to \R$ and $\Var_A(f): \Omega \to \R$  to be the conditional expectation and variance of $f$ given the configuration on $T \setminus A$; i.e., for $\xi \in \Omega$, 
$(\Ex_A(f))(\xi)$ and $(\Var_A(f))(\xi)$ are expectation and variance of $f$ on $A$ given $\xi_{T \setminus A}$ outside $A$:
\begin{align*}
(\Ex_A(f))(\xi) &= \Ex_A^\xi(f) = \Ex_\mu[f \mid \sigma_{T\setminus A} = \xi_{T\setminus A}],~\text{and}\\
(\Var_A(f))(\xi) &= \Var_A^\xi(f) = \Var_\mu[f \mid \sigma_{T\setminus A} = \xi_{T\setminus A}].
\end{align*}
Observe that both $\Ex_A(f)$ and $\Var_A(f)$ depend only on the configuration on $T \setminus A$. 
Furthermore, we will write $\Ex(f) = \Ex_T(f) = \Ex_\mu(f)$ and $\Var(f) = \Var_T(f) = \Var_\mu(f)$ for convenience. 

We compile next several useful, standard properties of the expectation and variance functionals that we shall use in our proofs.

\begin{lemma}\label{lemma:e-var}
	Let $\eta \in \Omega$ and $f:\Omega \to \R$ be an arbitrary function and
	\begin{enumerate}
		\item (Law of total expectation)
		For every $A \subseteq F \subseteq T$, we have
		\[
		\Ex_F^\eta(f) = \Ex_F^\eta(\Ex_A(f)).
		\]
		\item (Law of total variance)
		For every $A \subseteq F \subseteq T$, we have
		\[
		\Var_F^\eta(f) = \Ex_F^\eta[\Var_A(f)] + \Var_F^\eta[\Ex_A(f)].
		\]
		\item (Convexity of variance)
		For every $A,B \subseteq T$ such that $A \cap B = \emptyset$ and there is no edge between $A,B$ (i.e., $\partial A \cap B = \emptyset = A \cap \partial B$), we have
		\[
		\Var_A^\eta[\Ex_B(f)] \le \Ex_B^\eta[\Var_A(f)].
		\]
	\end{enumerate}
\end{lemma}
For proofs of these facts see, e.g.,~\cite{MSW04} and the references therein, but for example note that part 1 follows directly from the definitions:
$$
\Ex_F^\eta(\Ex_A(f)) = \sum_{\xi \in [q]^{F \setminus A}} \mu_F^\eta(\xi) \Ex_A(f)(\xi) = \sum_{\xi \in [q]^{F \setminus A}} \sum_{\gamma \in [q]^{A}} \mu_F^\eta(\xi) \mu_F^\eta(\gamma\mid\xi)f(\xi,\gamma) 
= \Ex_F^\eta(f).
$$

The Dirichlet form of the heat-bath block dynamics satisfies
\begin{equation}\label{eq:calE}
\mathcal{E}_{\Ptb}(f,f) = \frac{1}{\ell+1} \cdot \sum_{j=1}^{\ell+1} \Ex[\Var_{T_j^\ell}(f)];
\end{equation}
see, e.g., Fact 3.3 in~\cite{BCSV}.
We present next two key lemmas. 

\begin{lemma}\label{lem:dirich-bound}
	For every function $f: \Omega \to \R$ we have
	\[
	\sum_{j=1}^{\ell+1} \Ex[\Var_{T_j^\ell}(f)] \ge \sum_{i=1}^{h+1} \Ex[\Var_{B_i^\ell}(\Ex_{F_{i-\ell-1}}(f))].
	\]
\end{lemma}

\begin{lemma}\label{lem:var-bound}
	If $\mu = \mu_T^\tau$ satisfies $\PVM(\ell,\eps)$ for $\eps = \frac{1-\delta}{2(\ell+1)}$, then for every function $f: \Omega \to \R$ we have
	\[
	\frac{2}{\delta} \cdot \sum_{i=1}^{h+1} \Ex[\Var_{B_i^\ell}(\Ex_{F_{i-\ell-1}}(f))] \ge\Var(f).
	\]
\end{lemma}

The proof of \cref{thm:PVM-gap} follows straightforwardly from these two facts.
\begin{proof}[Proof of \cref{thm:PVM-gap}]
	Lemmas \ref{lem:dirich-bound} and \ref{lem:var-bound} combined with \eqref{eq:calE} imply:	
	$$
	\mathcal{E}_{\Ptb}(f,f) \ge \frac{\delta}{2(\ell+1)} \Var(f).
	$$
	The result then follows from~\eqref{eq:sw:gap}.
\end{proof}

We provide next the proof of \cref{lem:dirich-bound}, which does not use PVM 
and exploits instead the recursive structure of the $d$-ary tree.

\begin{proof}[Proof of \cref{lem:dirich-bound}]
	Fix $\ell$ and $j$, 
	and recall the definition of the blocks $T_j^\ell$ and $B_{i}^\ell$; see~\eqref{eq:tj},~\eqref{eq:bi} and Figures~\ref{fig:tiling}(b) and~\ref{fig:tiling}(c).
	Observe that $T_j^\ell$ is the union of $B_{i}^\ell$'s for certain sequence of $i$'s. 
	Specifically,
	\[
	T_j^\ell = \bigcup_{k=0}^m B_{i(k)}^\ell.
	\]
	where $i(k) = i_j^\ell(k) = k(\ell+1)+j$ for $k=0,1,\dots,m$ and 
	$m=m_j^\ell$ is the smallest positive integer such that $k(\ell+1)+j \ge h+1$. 
	
	For $0\le k \le m$ we define
	\[
	S_k = (S_j^\ell)_k = \bigcup_{r=0}^k B_{i(r)}^\ell.
	\]
	Note that $S_m = T_j^\ell$, $S_0 = B_j^\ell$, $S_k = S_{k-1} \cup B_{i(k)}^\ell$, and $S_k \subseteq F_{i(k)}$.
	Then, for any $\eta \in \Omega$ the law of total variance (see \cref{lemma:e-var})  implies:
	$$
	\Var_{S_k}^\eta(f) = \Ex_{S_k}^\eta[\Var_{S_{k-1}}(f)] + \Var_{S_k}^\eta[\Ex_{S_{k-1}}(f)].
	$$
	Averaging over $\eta$ and using the law of total expectation (see \cref{lemma:e-var}) we deduce
	$$
	\Ex[\Var_{S_k}(f)] = \Ex[\Var_{S_{k-1}}(f)] + \Ex[\Var_{S_k}(\Ex_{S_{k-1}}(f))].
	$$
	Similarly, we deduce 
	$$
	\Ex[\Var_{S_k}(\Ex_{S_{k-1}}(f))] = \Ex\left[ \Var_{B_{i(k)}^\ell}(\Ex_{S_{k-1}}(f)) \right] +
	\Ex\left[ \Var_{S_k} \left( \Ex_{B_{i(k)}^\ell}(\Ex_{S_{k-1}}(f)) \right) \right],
	$$
	and so
	\begin{align}
	\Ex[\Var_{S_k}(f)]
	&= \Ex[\Var_{S_{k-1}}(f)] +
	\Ex\left[ \Var_{B_{i(k)}^\ell}(\Ex_{S_{k-1}}(f)) \right] +
	\Ex\left[ \Var_{S_k} \left( \Ex_{B_{i(k)}^\ell}(\Ex_{S_{k-1}}(f)) \right) \right] \notag\\
	&\ge \Ex[\Var_{S_{k-1}}(f)] +
	\Ex\left[ \Var_{B_{i(k)}^\ell}(\Ex_{S_{k-1}}(f)) \right]. \label{eq:var-bound} 
	\end{align}
	The sets $B_{i(k)}^\ell$ and $F_{i(k-1)}$ are at distance $2$ from each other,
	so the convexity of variance from \cref{lemma:e-var}
	implies that 
	\begin{align*}
	\Ex\left[ \Var_{B_{i(k)}^\ell} (\Ex_{S_{k-1}}(f)) \right] 
	&= \Ex\left[  \Ex_{F_{i(k)-\ell-1}}\left[ \Var_{B_{i(k)}^\ell} (\Ex_{S_{k-1}}(f)) \right] \right]\\
	&\ge \Ex\left[ \Var_{B_{i(k)}^\ell} \left( \Ex_{F_{i(k)-\ell-1}}(\Ex_{S_{k-1}}(f)) \right) \right]
	= \Ex\left[ \Var_{B_{i(k)}^\ell} (\Ex_{F_{i(k)-\ell-1}}(f)) \right].
	\end{align*}
	Plugging this bound into~\eqref{eq:var-bound} we obtain for any integer $k \ge 1$ that
	\begin{equation}\label{eq:fb}
	\Ex[\Var_{S_k}(f)] 
	\ge \Ex[\Var_{S_{k-1}}(f)] +
	\Ex\left[ \Var_{B_{i(k)}^\ell}(\Ex_{F_{i(k)-\ell-1}}(f))  \right].
	\end{equation}
	When $k=0$ we let $S_{-1} = \emptyset$, and it is straightforward to check that everything above still holds trivially.	
	Since $S_m = T_j^\ell$, we derive from~\eqref{eq:fb} that
	\begin{align*}
	\Ex[\Var_{T_j^\ell}(f)]
	&= \Ex[\Var_{S_m}(f)] - \Ex[\Var_{S_{-1}}(f)]\\
	&= \sum_{k=0}^{m} \Ex[\Var_{S_k}(f)] - \Ex[\Var_{S_{k-1}}(f)]\\
	&\ge \sum_{k=0}^{m} \Ex\left[ \Var_{B_{i(k)}^\ell}(\Ex_{F_{i(k)-\ell-1}}(f)) \right].
	\end{align*}
	Hence, 
	summing over $j$
	\begin{align*}
	\sum_{j=1}^{\ell+1}\Ex[\Var_{T_j^\ell}(f)]
	\ge 
	\sum_{i=1}^{h+1} \Ex\left[ \Var_{B_{i}^\ell}(\Ex_{F_{i-\ell-1}}(f)) \right],
	\end{align*}
	as claimed.
\end{proof}

%
%

It remains for us to establish \cref{lem:var-bound}.
The following lemma will be helpful.
\begin{lemma}\label{lem:PVM-varbound}
	Let $F = A\cup B \subseteq T$ and $\eta \in \Omega$. 
	Suppose that for every function $g:\Omega_F^\eta \to \R$ that is independent of $B$, we have $\Var_F^\eta[\Ex_A(g)] \le \eps \cdot \Var_F^\eta(g)$ for some constant $\eps \in (0,1/2)$.
	Then for every function $f:\Omega_F^\eta \to \R$ we have
	\begin{equation}
	\label{eq:sl}
	\Var_F^\eta[\Ex_A(f)] \le
	\frac{2(1-\eps)}{1-2\eps} \cdot \Ex_F^\eta[\Var_B(f)] +
	\frac{2\eps}{1-2\eps} \cdot \Ex_F^\eta[\Var_A(f)];
	\end{equation}
	In addition for $A' \subseteq A$ we have
	\[
	\Var_F^\eta[\Ex_A(f)] \le
	\frac{2(1-\eps)}{1-2\eps} \cdot \Ex_F^\eta[\Var_B(\Ex_{A'}(f))] +
	\frac{2\eps}{1-2\eps} \cdot \Ex_F^\eta[\Var_A(\Ex_{A'}(f))].
	\]
\end{lemma}

\begin{proof}
	The first part was established in the proof of Lemma 3.5 from \cite{MSW04}.
	For the second part, 
	note that the law of total expectation from \cref{lemma:e-var},
	$\Var_F^\eta[\Ex_A(f)] = \Var_F^\eta[\Ex_A(\Ex_{A'}(f))]$.
	Replacing $f$ by $\Ex_{A'}(f)$ in~\eqref{eq:sl} yields the result.
\end{proof}

We present next the proof of \cref{lem:var-bound}. 

\begin{proof}[Proof of \cref{lem:var-bound}]
	By the law of total variance (\cref{lemma:e-var}), we deduce that for each $1\le i \le h+1$ and $\eta \in \Omega$,
	\begin{equation*}
	\Var_{F_i}^\eta(f)
	= \Ex_{F_i}^\eta[\Var_{F_{i-1}}(f)] + \Var_{F_i}^\eta(\Ex_{F_{i-1}}(f)).
	\end{equation*}
	Taking expectations we obtain:
	\begin{equation}
	\Ex[\Var_{F_i}(f)]
	= \Ex[\Var_{F_{i-1}}(f)] + \Ex[\Var_{F_i}(\Ex_{F_{i-1}}(f))].
	\label{eq:telescope}
	\end{equation}
	Recall that $F_0 = \emptyset$ and $F_{h+1} = T$. Then,
	\begin{align}
	\Var(f) &= \Ex[\Var_{F_{h+1}}(f)] - \Ex[\Var_{F_0}(f)] \nonumber\\
	&= \sum_{i=1}^{h+1} \Ex[\Var_{F_i}(f)] - \Ex[\Var_{F_{i-1}}(f)] \nonumber\\
	&= \sum_{i=1}^{h+1} \Ex[\Var_{F_i}(\Ex_{F_{i-1}}(f))]. \label{eq:var(f)}
	\end{align}
	Now, since $\mu$ satisfies $\PVM(\ell,\eps)$, it follows from \cref{lem:PVM-varbound} (with $F = F_i$, $A = F_{i-1}$, $B = B_i^\ell$, and $A' = F_{i-\ell-1}$ and taking expectation on both sides) that for every $1\le i \le h+1$,
	\begin{equation}\label{eq:lemma-bound}
	\Ex[\Var_{F_i}(\Ex_{F_{i-1}}(f))]
	\le
	\frac{2(1-\eps)}{1-2\eps} \cdot \Ex[\Var_{B_i^\ell}(\Ex_{F_{i-\ell-1}}(f))] +
	\frac{2\eps}{1-2\eps} \cdot \Ex[\Var_{F_{i-1}}(\Ex_{F_{i-\ell-1}}(f))].
	\end{equation}
	Let $g = \Ex_{F_{i-\ell-1}}(f)$ and observe that $\Var_{F_{i-\ell-1}}^\eta(g) = 0$ for all $\eta \in \Omega$ as $g$ is independent of $F_{i-\ell-1}$.
	Then,
	\begin{align*}
	\Ex[\Var_{F_{i-1}}(\Ex_{F_{i-\ell-1}}(f))]
	= \Ex[\Var_{F_{i-1}}(g)] - \Ex[\Var_{F_{i-\ell-1}}(g)]
	= \sum_{j=i-\ell}^{i-1} \Ex[\Var_{F_j}(g)] - \Ex[\Var_{F_{j-1}}(g)].
	\end{align*}
	By \eqref{eq:telescope} (which holds trivially for $i\le 0$ as well) and the law of total expectation we deduce that
	\begin{align}
	\label{eq:msw01}
	\Ex[\Var_{F_{i-1}}(\Ex_{F_{i-\ell-1}}(f))]
	= \sum_{j=i-\ell}^{i-1} \Ex[\Var_{F_j}(\Ex_{F_{j-1}}(g))]
	= \sum_{j=i-\ell}^{i-1} \Ex[\Var_{F_j}(\Ex_{F_{j-1}}(f))].
	\end{align}	
	Therefore, we get from \eqref{eq:msw01} and \eqref{eq:var(f)} that
	\begin{align}
	\sum_{i=1}^{h+1} \Ex[\Var_{F_{i-1}}(\Ex_{F_{i-\ell-1}}(f))]
	&= \sum_{i=1}^{h+1} \sum_{j=i-\ell}^{i-1} \Ex[\Var_{F_j}(\Ex_{F_{j-1}}(f))] \notag\\
	&\le \ell \cdot \sum_{i=1}^{h+1} \Ex[\Var_{F_i}(\Ex_{F_{i-1}}(f))] 
	= \ell \cdot \Var(f).
	\label{eq:mu-var-mu}
	\end{align}
	Combining \eqref{eq:var(f)}, \eqref{eq:lemma-bound} and \eqref{eq:mu-var-mu} we get
	\begin{align*}
	\Var(f) &= \sum_{i=1}^{h+1} \Ex[\Var_{F_i}(\Ex_{F_{i-1}}(f))]\\
	&\le \frac{2(1-\eps)}{1-2\eps} \cdot \sum_{i=1}^{h+1} \Ex[\Var_{B_i^\ell}(\Ex_{F_{i-\ell-1}}(f))]
	+ \frac{2\eps}{1-2\eps} \cdot \sum_{i=1}^{h+1} \Ex[\Var_{F_{i-1}}(\Ex_{F_{i-\ell-1}}(f))]\\
	&\le \frac{2(1-\eps)}{1-2\eps} \cdot \sum_{i=1}^{h+1} \Ex[\Var_{B_i^\ell}(\Ex_{F_{i-\ell-1}}(f))]
	+ \frac{2\eps}{1-2\eps} \cdot \ell \cdot \Var(f).
	\end{align*}
	We then conclude that
	\[
	\sum_{i=1}^{h+1} \Ex[\Var_{B_i^\ell}(\Ex_{F_{i-\ell-1}}(f))] \ge \frac{1-2\eps(\ell+1)}{2(1-\eps)} \cdot \Var(f) \ge \frac{\delta}{2} \cdot \Var(f). \qedhere
	\]
\end{proof}

\section{Entropy Mixing implies fast mixing: Proof of \texorpdfstring{\cref{thm:entropy-main}}{Theorem 2}}
\label{sec:ent-proof}

Let $E \subseteq T$ denote the set of all \emph{even} vertices of the tree $T$, where a vertex is called even if its distance to the leaves is even; let $O = T\setminus E$ be the set of all odd vertices. 
We show that EM (i.e., entropy mixing) as defined in~\cite{MSW04} implies 
a factorization of entropy into even and odd subsets of vertices. 
This even-odd factorization was recently shown  to imply $O(\log n)$ mixing
of the SW dynamics on bipartite graphs~\cite{BCPSVstoc}.

We start with the definition of EM, which is the analog of the $\VM$ condition for entropy.
Let $\tau$ be a fixed boundary condition and again set
$\mu := \mu_T^\tau$ and $\Omega := \Omega^\tau$ for ease of notation.
Recall that for $v\in T$, we use $T_v$ for the subtree of $T$ rooted at $v$. 
Recall that for $\eta \in \Omega$ and $g: \Omega_{T_v}^\eta \to \R$, we defined the function
$g_v(a) = \E_{T_v}^\eta[g \mid \sigma_v = a]$ for $a \in [q]$; see \eqref{eq:g}.

\begin{definition}[Entropy Mixing (EM)]\label{def:em}
	The Gibbs distribution $\mu = \mu_T^\tau$ satisfies $\EM(\ell,\eps)$ if for every $v\in T$, every $\eta \in \Omega$, and every function $g: \Omega_{T_v}^\eta \to \R$ that is independent of the configuration on $B(v,\ell)$, we have
	$
	\Ent_{T_v}^\eta(g_v) \le \eps \cdot \Ent_{T_v}^\eta(g).
	$
	The EM condition holds if there exist constants $\ell$ and $\eps = \eps(\ell)$ such that $\EM(\ell,\eps)$ holds.	
\end{definition}

Extending our notation from the previous section for the variance functional,
for $A \subseteq T$ 
and a function $f: \Omega \to \R_{\ge 0}$,  
we use $\Ent_A(f)$ for the conditional entropy
of $f$ w.r.t.\ $\mu$ given a spin configuration in $T \setminus A$; 
i.e., for $\xi \in \Omega$ we have
\[
(\Ent_A(f))(\xi) = \Ent_A^\xi(f) = \Ent_\mu[f \mid \sigma_{T \setminus A} = \xi_{T\setminus A}].
\]
In particular, we shall write $\Ent(f) = \Ent_T(f) = \Ent_\mu(f)$. 
Notice that $\Ent_A(f)$ can be viewed as a function from $[q]^{T \setminus A}$ to $\R_{\ge 0}$
and $\Ex[\Ent_A(f)]$ denotes its mean, averaging over the configuration on $T \setminus A$.
We state next our even-odd factorization of entropy.

\begin{theorem}\label{thm:EM-EOF}
	If there exist $\ell\in\N^+$ and $\eps \in(0,1)$ such that $\mu = \mu_T^\tau$ satisfies $\EM(\ell,\eps)$, then 
	there exists a constant $C_{\textsc{eo}} = C_{\textsc{eo}}(\ell,\eps)$ independent of $n$ such that 
	for every function $f: \Omega \to \R_{\ge 0}$ we have 
	$
	\Ent(f) \le C_{\textsc{eo}} \left( \Ex[\Ent_E(f)] + \Ex[\Ent_O(f)] \right).
	$
	
\end{theorem}

\cref{thm:entropy-main} follows immediately.

\begin{proof}[Proof of \cref{thm:entropy-main}]
	By \cref{thm:EM-EOF}, EM implies the even-odd factorization of entropy,
	and the results in \cite{BCPSVstoc} imply that the mixing time of the SW dynamics is $O(\log n)$. 
\end{proof}
Our main technical contribution in the proof~\cref{thm:entropy-main}
is thus~\cref{thm:EM-EOF}; namely, that EM implies the even-odd factorization of entropy.
To prove \cref{thm:EM-EOF}, we will first establish entropy factorization for the tiled blocks defined in \eqref{eq:tj} and \eqref{eq:bi}; see also Figures~\ref{fig:tiling}(b) and~\ref{fig:tiling}(c). 
From the tiled block factorization of entropy we then deduce the desired even-odd factorization. 
This approach is captured by the following two lemmas.

\begin{lemma}\label{lem:EM-TBF}
	If there exist $\ell\in\N^+$ and $\eps\in(0,1)$ such that $\mu = \mu_T^\tau$ satisfies $\EM(\ell,\eps)$, then 
	there exists a constant $C_{\textsc{tb}} = C_{\textsc{tb}}(\ell,\eps)$ independent of $n$ such that, for every function $f: \Omega \to \R_{\ge 0}$,
	$
	\Ent(f) \le C_{\textsc{tb}} \cdot \sum_{j = 1}^{\ell + 1} \Ex[\Ent_{T_j^\ell}(f)].
	$
\end{lemma}

\begin{lemma}\label{lem:TBF-EOF}
	If for every function $f: \Omega \to \R_{\ge 0}$ we have
	$
	\Ent(f) \le C_{\textsc{tb}} \cdot \sum_{j = 1}^{\ell + 1} \Ex[\Ent_{T_j^\ell}(f)], 
	$
	then there exists  $C_{\textsc{eo}} = C_{\textsc{eo}}(C_{\textsc{tb}},\ell)$ such that for every function $f: \Omega \to \R_{\ge 0}$ we have
	\[
	\Ent(f) \le C_{\textsc{eo}} \left( \Ex[\Ent_E(f)] + \Ex[\Ent_O(f)] \right).
	\]
\end{lemma}

\begin{proof}[Proof of \cref{thm:EM-EOF}]
	Follows directly from \cref{lem:EM-TBF,lem:TBF-EOF}.
\end{proof}	
We proved a version of~\cref{lem:EM-TBF} for the variance functional as part of the proof of~\cref{thm:PVM-gap}, and the same argument can then be easily adapted to entropy; its proof is provided in~\cref{subsec:em} . We provide next the proof of \cref{lem:TBF-EOF}, which contains the main novelty in our proof of~\cref{thm:EM-EOF}.

\begin{proof}[Proof of \cref{lem:TBF-EOF}]
	First, we claim that there exists a constant $C' = C'(\ell)$ such that for every function $f: \Omega^\eta_{B(v,\ell)} \to \R_{\ge 0}$ one has
	the following inequality: 
	\begin{equation}\label{eq:eo-crude}
	\Ent^\eta_{B(v,\ell)}(f)
	\le C' \left( \Ex^\eta_{B(v,\ell)}[\Ent_{B(v,\ell) \cap E}(f)] + \Ex^\eta_{B(v,\ell)}[\Ent_{B(v,\ell) \cap O}(f)] \right).
	\end{equation}
	To deduce~\eqref{eq:eo-crude}, consider the even-odd block dynamics $M$ in $B(v,\ell)$ with boundary condition $\eta$ and blocks $\mathcal U = \{E \cap B(v,\ell),O \cap B(v,\ell)\}$.
	A simple coupling argument implies that the spectral gap of $M$ is $\Omega(1)$. 
	Then, Corollary A.4 from \cite{DSC} implies that the log-Sobolev constant $\alpha(M)$ of $M$ is $\Omega(1)$, which establishes \eqref{eq:eo-crude} with constant $C' = O(1/\alpha(M))$. 
	We note that all bounds and comparisons in this argument are fairly crude, and, in fact, the constant $C'$ depends exponentially on $|B(v,\ell)|$, but it is still independent of $n$.

	Next, notice that, for any $\eta \in \Omega$, $\mu_{T_j^\ell}^\eta$ is the product of a collection of distributions on (disjoint) subsets $B(v,\ell)$. 
	Lemma 3.2 from \cite{CP}
	allows us to lift the ``local'' even-odd factorization in each $B(v,\ell)$ from~\eqref{eq:eo-crude} to a ``global'' even-odd factorization in $T_j^\ell$.
	Specifically, for every function $f: \Omega_{T_j^\ell}^\eta \to \R_{\ge 0}$ we obtain
	\[
	\Ent_{T_j^\ell}^\eta(f)
	\le C' \left( \Ex_{T_j^\ell}^\eta[ \Ent_{T_j^\ell \cap E} (f) ] + \Ex_{T_j^\ell}^\eta[ \Ent_{T_j^\ell \cap O} (f) ] \right). 
	\]
	Taking expectation over $\eta$, we get
	\begin{align*}
	\Ex[\Ent_{T_j^\ell}(f)] &\le C' \left( \Ex[\Ent_{T_j^\ell \cap E} (f)] + \Ex[\Ent_{T_j^\ell \cap O} (f)] \right)\le C' \left( \Ex[\Ent_E (f)] + \Ex[\Ent_O (f)] \right);
	\end{align*}
	the last inequality follows from the fact
	that $\Ent_E^\eta(f) = \Ex_E^\eta[\Ent_{T_j^\ell \cap E}(f)] + \Ent_E^\eta[\Ex_{T_j^\ell \cap E}(f)].$
	Summing up over $j$, we obtain 
	\begin{align*}
	\sum _{j=1}^{\ell+1}\Ex[\Ent_{T_j^\ell}(f)] &\le C'(\ell+1) \left( \Ex[\Ent_E (f)] + \Ex[\Ent_O (f)] \right),
	\end{align*}
	and the result follows by taking $C_{\textsc{eo}} = C' (\ell+1)$.
\end{proof}

\subsection{Proof of ~\cref{lem:EM-TBF}}
\label{subsec:em}

We provide next the proof of~\cref{lem:EM-TBF}, which follows from the next two lemmas.

\begin{lemma}\label{lem:ent1-bound}
	For all $f: \Omega \to \R_{\ge 0}$,
	$
	\sum_{j=1}^{\ell+1} \Ex[\Ent_{T_j^\ell}(f)] \ge \sum_{i=1}^{h+1} \Ex[\Ent_{B_i^\ell}(\Ex_{F_{i-\ell-1}}(f))].
	$
\end{lemma}

\begin{lemma}\label{lem:ent2-bound}
	If $\mu = \mu_T^\tau$ satisfies $\EM(\ell,\eps)$, then 
	there exists a constant $C = C(\ell ,\eps)$ such that
	for every function $f: \Omega \to \R_{\ge 0}$ we have
	$
	\Ent(f) \le C \cdot \sum_{i=1}^{h+1} \Ex[\Ent_{B_i^\ell}(\Ex_{F_{i-\ell-1}}(f))].
	$
	
\end{lemma}

\begin{proof}[Proof of \cref{lem:EM-TBF}]
	Follows from \cref{lem:ent1-bound,lem:ent2-bound}.
\end{proof}

\cref{lem:ent1-bound,lem:ent2-bound} are counterparts of \cref{lem:dirich-bound,lem:var-bound}, respectively, for entropy. 
In particular, the proof of \cref{lem:ent1-bound} is identical to that of \cref{lem:dirich-bound}, replacing variance by entropy everywhere and is thus omitted.
Note that
the properties in \cref{lemma:e-var} hold for entropy as well (see \cite{MSW04}). 

It remains to prove \cref{lem:ent2-bound}, but again its proof is almost the same as \cref{lem:var-bound} (replacing variance with entropy). 
We only require the following lemma to play the role of \cref{lem:PVM-varbound} in
the proof of \cref{lem:var-bound}.
Let $p_{\min}$ denote the minimum probability of any vertex receiving any spin value under any neighborhood configuration; then, $p_{\min} \ge \frac{1}{q} e^{-\beta (d+1)}$. 

\begin{lemma}\label{lem:EM-entbound}
	For any $\eps < p_{\min}^2$, if the Gibbs distribution $\mu = \mu_T^\tau$ satisfies $\EM(\ell,\eps)$, 
	then for $\eps' = \frac{\sqrt{\eps}}{p_{\min}}$, every $1 \le i \le h+1$, every $\eta \in \Omega$, and every function $f: \Omega_{F_i}^\eta \to \R_{\ge 0}$, we have
	\[
	(1-\eps') \Ent_{F_i}^\eta(f) \le \Ex_{F_i}^\eta [ \Ent_{B_i^\ell}(f) ] + \Ex_{F_i}^\eta [ \Ent_{F_{i-1}}(f) ].
	\]
\end{lemma}

\begin{proof}
	As shown by Lemma 3.5(ii) of \cite{MSW04}, $\EM(\ell,\eps)$ implies that for every $v\in T$, every $\eta \in \Omega$, and every function $f: \Omega_{T_v}^\eta \to \R_{\ge 0}$, we have for $\eps' = \frac{\sqrt{\eps}}{p_{\min}}$
	\[
	(1-\eps') \Ent_{T_v}^\eta(f) \le \Ex_{T_v}^\eta [ \Ent_{B(v,\ell)}(f) ] + \Ex_{T_v}^\eta [ \Ent_{T_v \setminus v}(f) ]
	\]
	This entropy factorization holds for every subtree, and in particular for 
	all the subtrees rooted at the same level (e.g., for all $v \in L_i$).
	Then, we can apply Lemma 3.2 from \cite{CP} to obtain such a factorization for $F_i$. 
	Specifically, for every $1 \le i \le h+1$, every $\eta \in \Omega$, and every function $f: \Omega_{F_i}^\eta \to \R_{\ge 0}$, we have from Lemma 3.2 in \cite{CP} that
	\[
	(1-\eps') \Ent_{F_i}^\eta(f) \le \Ex_{F_i}^\eta [ \Ent_{B_i^\ell}(f) ] + \Ex_{F_i}^\eta [ \Ent_{F_{i-1}}(f) ], 
	\]
	since
	$F_i = \bigcup_{v\in L_i} T_v$,
	$F_{i-1} = \bigcup_{v\in L_i} T_v \setminus v$, and
	$B_i^\ell = \bigcup_{v\in L_i} B(v,\ell)$. 
\end{proof}

\cref{lem:ent2-bound} then can be proved in the same way as \cref{lem:var-bound}, simply using \cref{lem:EM-entbound} instead of \cref{lem:PVM-varbound}.

\begin{remark}
	\label{rmk:proof:cmp}
	In \cref{sec:variance}, we establish bounds on the spectral gap of the tiled block dynamics under the VM condition. This is equivalent to the tiled block factorization of variance. The schematic of our proof is:
	\[
	\parbox{40pt}{\centering VM} 
	\overset{\text{\cref{thm:VM=PVM}}}{\Longrightarrow} 
	\parbox{40pt}{\centering PVM} 
	\overset{\substack{\text{Lemma 3.5(i)} \\ \text{\cite{MSW04}}}}{\Longrightarrow} 
	\parbox{80pt}{\centering Parallel\\Variance\\Factorization} 
	\overset{\text{\cref{lem:dirich-bound,lem:var-bound}}}{\Longrightarrow} 
	\parbox{80pt}{\centering Tiled Block\\Variance\\Factorization} 
	\]
	Our proof in this section for the tiled block factorization of entropy, while similar, follows a slightly different route:
	\[
	\parbox{40pt}{\centering EM} 
	\overset{\substack{\text{Lemma 3.5(ii)} \\ \text{\cite{MSW04}}}}{\Longrightarrow} 
	\parbox{60pt}{\centering Entropy\\Factorization} 
	\overset{\substack{\text{Lemma 3.2} \\ \text{\cite{CP}}}}{\Longrightarrow} 
	\parbox{80pt}{\centering Parallel\\Entropy\\Factorization} 
	\overset{\text{\cref{lem:ent1-bound,lem:ent2-bound}}}{\Longrightarrow} 
	\parbox{80pt}{\centering Tiled Block\\Entropy\\Factorization} 
	\]
\end{remark}

\section{Comparison between block dynamics and the SW dynamics}
\label{sec:sw}

In this section we bound the spectral gap of the SW dynamics in terms of the gap of the tiled-block dynamics.
We do so in a general setting, i.e., for arbitrary graphs, block dynamics, and boundary conditions; in particular, we prove~\Cref{thm:sw-block}.
The proofs in this section extend ideas from~\cite{BCSV,Ullrich}.
We believe our generalization could find useful applications in the future.

Let $G = (V\cup \partial V,E)$ be a graph. 
We assume $V\cap \partial V = \emptyset$ 
and interpret $\partial V$ as the boundary of~$V$.
Define $\tau$ to be a fixed spin configuration on $\partial V$ viewed as a boundary condition.
Let $\mu^\tau_G$ be the Potts distribution on $G$ with boundary condition $\tau$ and let $\Omega_G^\tau$ be the set of Potts configurations of $G$ consistent with $\tau$.

Given a Potts configuration $\sigma_t \in \Omega_G^\tau$ at time $t$, the SW dynamics generates the next configuration $\sigma_{t+1}$ as follows:	
\begin{enumerate}
	\item Obtain $A_t \subseteq E$ by including each monochromatic edge of $E$ in $\sigma_t$ independently with probability $p$; 
	\item For each connected component $C$ of the graph $(V\cup \partial V, A_t)$ such that $C \subseteq V$ (i.e., those containing \emph{no} vertices from the boundary $\partial V$), we pick a new spin from $\{1,\dots,q\}$ u.a.r.\ and assign it to every vertex of $C$; 
	Vertices from other components keep their spin in $\sigma_t$. 
\end{enumerate}
Observe that the boundary condition $\tau$ determines the spin of all the vertices connected to $\partial V$ in $A_t$. The SW dynamics is reversible with respect to $\mu_G^\tau$; see, e.g.,~\cite{ES}.

We introduce next a block variant of the SW dynamics.
Let $\mathcal D = \{D_1,\dots,D_m\}$ be such that $D_i \subseteq V$ and $\cup_{i=1}^m D_i = V$.
Given a configuration $\sigma_t$:
\begin{enumerate}
	\item Obtain $A_t \subseteq E$ by including each monochromatic edge of $E$ in $\sigma_t$ independently with probability $p$; 
	\item Pick a random block $D_i$ from $\mathcal D$;
	\item For each connected component $C$ of the graph $(V\cup \partial V, A_t)$ such that $C \subseteq D_i$ (i.e., those containing \emph{no} vertices from $V\setminus D_i$ or $\partial V$), we pick a new spin from $\{1,\dots,q\}$ u.a.r.\ and assign it to every vertex of $C$; 
	Vertices from other components keep their spin in $\sigma_t$. 
\end{enumerate}
Let $\SWT$ denote the transition matrix of this chain; we shall see that
$\SWT$ is also reversible w.r.t.\ $\mu_G^\tau$.
Recall that $\PSW$ denotes the transition matrix for the SW dynamics.
We prove the following.
\begin{lemma} 
	\label{lem:SW-comp1}	
	For every function $f:\Omega_G^\tau \to \R$, we have
	$
	\mathcal{E}_{\PSW}(f,f) \ge \mathcal{E}_{\SWT}(f,f).
	$
\end{lemma}


\begin{proof}
	This proof uses a decomposition of the transition matrices
	$\PSW$ and $\SWT$ as products of simpler matrices introduced by Ullrich~\cite{Ullrich}.
	Let $\JC \subseteq 2^E \times \Omega_G^\tau$ denote the joint space, 
	where each configuration is a pair $(A,\sigma)$ such that $A \subseteq E$ and $\sigma \in \Omega_G^\tau$. 		
	The joint Edwards-Sokal measure~\cite{ES} on $\JC$ is given by
	\begin{equation}\label{eq:joint}
	\nu_G^\tau(A,\sigma) = \frac{1}{Z_\textsc{J}} p^{|A|}(1-p)^{|E \setminus A|} \1 (A \subseteq M(\sigma)),
	\end{equation}
	where
	$p=1-{e}^{-\beta}$,
	$A \subseteq E$,
	$\sigma \in \Omega_G$,
	$M(\sigma)$ denotes the set of monochromatic edges of $E$ in $\sigma$,
	and $Z_\textsc{J}$ is the corresponding partition function. 	
	
	Let $A,B \subseteq E$ and $\sigma,\eta \in \Omega_G^\tau$.
	We define the matrices  $T$, $T^*$, $R$ and $Q_k$ with entries given by
	\begin{align*}
	T(\sigma,(A,\eta)) &= \1(\sigma=\eta)\1(A \subseteq M(\sigma)) \cdot p^{|A|} (1-p)^{|M(\sigma)\setminus A|}\\
	T^*((A,\eta),\sigma) &= \1(\eta = \sigma)\\
	R((A,\sigma),(B,\eta)) &= \1(A = B)\1(A \subseteq M(\sigma) \cap  M(\eta)) \cdot q^{-c(A)}\\
	Q_{k}((A,\sigma),(B,\eta)) &= \1(A = B)\1(A \subseteq M(\sigma) \cap  M(\eta))  \1(\sigma(V \setminus D_k)=\eta(V\setminus D_k)) \cdot q^{-c_k(A)},\notag
	\end{align*}
	where $c(A)$ is the number of connected components of $(V\cup \partial V,A)$ that are fully contained in $V$,
	and $c_k(A)$ is the number of those fully contained in $D_k$. 
	Notice that in the definition of $Q_k$, the condition $\sigma(V \setminus D_k)=\eta(V\setminus D_k)$ implies that every component containing a vertex from $V \setminus D_k$ has the same spin in $\sigma$ and $\eta$. 
	Then, we have the decomposition $\PSW = TRT^*$ and $\SWT = \frac{1}{m}\sum_{i=1}^m TQ_kT^*$ following from the definition.
	Note that $T$ is a $|\Omega_G^\eta| \times |\JC|$ matrix,
	$T^*$ is a $|\JC| \times |\Omega_G^\eta|$ matrix, while $R$ and $Q_k$ have dimensions $|\JC| \times |\JC|$.
	
	The matrix $T$ defines an operator from $L_2(\nu_G^\tau)$ to $L_2(\mu_G^\tau)$.
	It is straightforward to check that  $T^*:L_2(\mu_G^\tau) \rightarrow L_2(\nu_G^\tau)$ is the adjoint of $T$.
	The matrices $R$ and $Q_k$ are self-adjoint operators from  $L_2(\nu_G^\tau)$ to $L_2(\nu_G^\tau)$
	and thus they are reversible w.r.t. $\nu_G^\tau$.
	(Note that this also implies that $\PSW$ and $\SWT$ are reversible w.r.t.\ $\mu_G^\tau$.)
	Moreover, since the matrices $R$ and $Q_k$ assign spins u.a.r.\ to components of a joint configuration,
	we have $R = Q_kRQ_k$ and $Q_k^2 = Q_k= Q_k^*$.
	
	From the definition of the Dirichlet form \eqref{eq:df}
	we have for every function $f: \Omega_G^\tau \to \R$ that
	$\mathcal{E}_{\PSW}(f,f) = \inner{f}{(I-\PSW)f}{\mu_G^\tau}$
	and
	$\mathcal{E}_{\SWT}(f,f)= \inner{f}{(I-\SWT) f}{\mu_G^\tau}$.
	Using the properties of adjoint operators we get for $f: \Omega_G^\tau \to \R$,
	\begin{align}
	\inner{f}{\PSW f}{\mu_G^\tau}
	&= \inner{f}{T R T^* f}{\mu_G^\tau}
	= \inner{f}{T Q_kRQ_k T^* f}{\mu_G^\tau}
	= \inner{Q_kT^*f}{RQ_k T^* f}{\nu_G^\tau} \notag \\
	&\le \inner{Q_kT^*f}{Q_k T^* f}{\nu_G^\tau}
	= \inner{f}{T Q_k^2 T^* f}{\mu_G^\tau}
	=	\inner{f}{T Q_k T^* f}{\mu_G^\tau}
	\label{eq:sw:first-comp}
	\end{align}
	where the inequality follows from the Cauchy-Schwarz inequality.
	Since this holds for every $k$, we get
	$
	\inner{f}{\PSW f}{\mu_G^\tau} \le \inner{f}{\SWT f}{\mu_G^\tau}
	$
	and the result follows.
\end{proof}

We consider next the standard heat-bath block dynamics with respect to $\mathcal D$.
We use $\mathcal B_{\mathcal D}$ to denote its transition matrix.
Note intuitively that $\mathcal B_{\mathcal D}$ should be faster then $\SWT$; conversely, we should be able to simulate one step of $\mathcal B_{\mathcal D}$ on a block $D_k$ by repeatedly performing the  corresponding move of $\SWT$ on $D_k$. The number of such moves should be related to the spectral gap of the block SW dynamics on $D_k$ with a fixed boundary condition on $V \setminus D_k$.

To formalize this intuition, we consider one additional variant of the block SW dynamics for a fixed block. 
For each $k\in[m]$ and $\eta \in \Omega_G^\tau$, let $\Omega_{D_k}^\eta \subseteq \Omega_G^\tau$ be the collection of configurations that agree with $\eta$ on $V \setminus D_k$, 
and let $\mu_{D_k}^\eta$ be the conditional Potts distribution over $\Omega_{D_k}^\eta$. 
Consider the following Markov chain 
such that if $\sigma_t \in\Omega_{D_k}^\eta$,
then $\sigma_{t+1} \in \Omega_{D_k}^\eta$ is obtained as follows:
\begin{enumerate}
	\item Obtain $A_t \subseteq E$ by including each monochromatic edge of $E$ in $\sigma_t$ independently with probability $p$; 
	\item For each connected component $C$ of the graph $(V\cup \partial V, A_t)$ such that $C \subseteq D_k$ (i.e., those containing \emph{no} vertices from $V\setminus D_k$ or $\partial V$), we pick a new spin from $\{1,\dots,q\}$ u.a.r.\ and assign it to every vertex of $C$; 
	Vertices from other components keep their spin in $\sigma_t$. 
\end{enumerate}
This Markov chain is ergodic and is reversible w.r.t.\ to $\mu_{D_k}^\eta$. 
Denote the transition matrix of this chain by $\SWTC{k}{\eta}$. 
Observe that, adopting the notations from the proof of \cref{lem:SW-comp1}, $\SWTC{k}{\eta}$ is the restriction of the transition matrix $T Q_k T^*$ to the subspace $\Omega_{D_k}^\eta$; i.e., we can write $\SWTC{k}{\eta} = (T Q_k T^*) |_{\Omega_{D_k}^\eta}$. 
(Note that $T Q_k T^*$ corresponds to a \emph{reducible} Markov chain with state space $\Omega_G^\tau$, where each strongly connected component of the state space is $\Omega_{D_k}^\eta$ for some $\eta$.) 
Let
$$
\gamma_{\mathrm{min}} = \min_{k=1,\dots,m} \min_{\eta \in \Omega_G^\tau} \GAP(\SWTC{k}{\eta}).
$$
We can show the following.
\begin{lemma} 
	\label{lemma:sw2}
	For every function $f:\Omega_G^\tau \to \R$, we have
	$
	\mathcal{E}_{\SWT}(f,f) \ge \gamma_{\mathrm{min}}  \cdot \mathcal{E}_{\mathcal B_{\mathcal D}}(f,f).
	$
\end{lemma}

\begin{proof}
	Let $f: \Omega_G^\tau \to \R$.
	From \eqref{eq:df} we get
	\begin{align*}
	{\mathcal E}_{\SWT}(f,f)
	&= \frac{1}{2} \sum_{\sigma,\sigma' \in \Omega_G^\tau} \mu_G^\tau(\sigma) \SWT(\sigma,\sigma') \cdot (f(\sigma) - f(\sigma'))^2\\
	&= \frac{1}{m} \sum_{k=1}^m \sum_{\eta \in \Omega_G^\tau} \mu_G^\tau(\eta) 
	\cdot \frac{1}{2} \sum_{\sigma,\sigma' \in \Omega_{D_k}^\eta} \mu_{D_k}^\eta(\sigma) \SWTC{k}{\eta}(\sigma,\sigma') \cdot (f(\sigma) - f(\sigma'))^2 \\
	&= \frac{1}{m} \sum_{k=1}^m \sum_{\eta \in \Omega_G^\tau} \mu_G^\tau(\eta) \cdot {\mathcal E}_{\SWTC{k}{\eta}}(f,f) \\
	&\geq \frac{1}{m} \sum_{k=1}^m \sum_{\eta \in \Omega_G^\tau} \mu_G^\tau(\eta) \cdot \GAP(\SWTC{k}{\eta}) \var{D_k}{\eta} (f) \\
	&\geq \gamma_{\mathrm{min}} \cdot \frac{1}{m} \sum_{k=1}^m \sum_{\eta \in \Omega_G^\tau} \mu_G^\tau(\eta) \var{D_k}{\eta} (f) \\
	&= {\gamma_{\mathrm{min}}} \cdot {\mathcal E}_{\mathcal B_\mathcal D}(f,f),
	\end{align*}
	as claimed.
\end{proof}

So far, we have not assumed anything about the geometry of the blocks in $\mathcal D$, so $\gamma_{\mathrm{min}}$ could be small (i.e., going to $0$ as $|V| \rightarrow \infty$). Our next result shows that $\gamma_{\textrm{min}} = \Omega(1)$ for a special class of block dynamics. 

Suppose that each block $D_k$ is such that $D_k = \cup_{j=1}^{\ell_k} D_{kj}$ where $\dist(D_{kj}, D_{kj'}) \ge 2$ for every $j \neq j'$.

This implies that 
for every $k$ and every $\eta \in \Omega_G^\tau$ the conditional Potts distribution 
$\mu_{D_k}^\eta$ is a product measure of all marginal distributions on each $D_{kj}$. 
We may write it as
$$\mu_{D_k}^\eta = \prod_{j=1}^{\ell_k} \mu_{D_{kj}}^\eta$$
where, with a slight abuse of notation, we view $\mu_{D_k}^\eta$ as a distribution over all configurations on $D_k$ (instead of over $\Omega_{D_k}^\eta$) and the same for $\mu_{D_{kj}}^\eta$'s. 
Let $\SWTC{kj}{\eta}$ be the transition matrix of the SW dynamics
on $D_{kj}$ with $\eta$ as the fixed boundary condition outside; so the stationary distribution of $\SWTC{kj}{\eta}$ is $\mu_{D_{kj}}^\eta$.

Note that since $D_k = \cup_{j=1}^{\ell_k} D_{kj}$ where $\dist(D_{kj}, D_{kj'}) \ge 2$ for every $j \neq j'$, after adding the edges in step 2  of $\SWTC{k}{\eta}$, every component contained in $D_k$
is fully contained in exactly one $D_{kj}$. Therefore,
$\SWTC{k}{\eta}$ is a product Markov chain of all $\SWTC{kj}{\eta}$'s; 
that is, each step of $\SWTC{k}{\eta}$ on $D_k$ is equivalent to applying one update of $\SWTC{kj}{\eta}$ on $D_{kj}$ simultaneously and independently for all $j$.

\begin{proposition}
	For $k \in [m]$ and $\eta \in \Omega(V \setminus D_k)$
	let $\sigma = (\sigma_1,\dots,\sigma_{\ell_k})\in \Omega(D_k)$,
	$\sigma' = (\sigma_1',\dots,\sigma_{\ell_k}')  \in \Omega(D_k)$
	where $\sigma_j,\sigma_j' \in \Omega(D_k^{j})$. Then
	$$
	\SWTC{k}{\eta}(\sigma,\sigma') = \prod_{j=1}^{\ell_k} \SWTC{kj}{\eta}(\sigma_j,\sigma_j').
	$$
\end{proposition}

The following is then a standard result
for the gap of product Markov chains (see, e.g., Lemma 4.7 in~\cite{BCSV}).

\begin{lemma} 
	\label{lemma:sw-3}
	For every $k \in [m]$ and $\eta \in \Omega_G^\tau$, 
	\[\GAP(\SWTC{k}{\eta}) = \min_{j=1,\dots,\ell_k} \GAP(\SWTC{kj}{\eta}).\]
\end{lemma}

Finally, we observe that for any configuration $\sigma \in \Omega_{D_{kj}}^\eta$,
in the percolation step of $\SWTC{kj}{\eta}$ from $\sigma$ every edge in $E_{kj}$ could be absent with probability
at least $1-p$, where $E_{kj}$ is the set of edges with at least one endpoint in $D_{kj}$.
Then, any two configurations $\sigma, \sigma' \in \Omega_{D_{kj}}^\eta$ can be coupled in one step with probability $\exp(-\beta |E_{kj}|)$. 
Thus, we obtain the following.
\begin{lemma} 
	\label{lemma:sw-4}
	There exists a constant $c := c(\beta)$ such that for every $k \in [m]$, $j \in [\ell_k]$, and $\eta \in \Omega_G^\tau$, 
	\[\GAP(\SWTC{kj}{\eta}) \ge \exp(-c |E_{kj}|).\]
\end{lemma}

Combining the results in this section, we can now prove our main comparison result in \cref{thm:sw-block}. 


\begin{proof}[Proof of \cref{thm:sw-block}]
	From \cref{lem:SW-comp1,lemma:sw2} it follows that
	for every function $f:\Omega_G^\tau \to \R$, we have
	$$
	\mathcal{E}_{\PSW}(f,f) \ge \gamma_{\mathrm{min}}  \cdot \mathcal{E}_{\mathcal B_{\mathcal D}}(f,f).
	$$
	The matrix $\mathcal B_{\mathcal D}$ is positive semidefinite since it is a heat-bath block dynamics (see, e.g., Fact 3.3 in~\cite{BCSV}). The transition matrix $\PSW$ is also positive semidefinite since as in \eqref{eq:sw:first-comp} we have
	\begin{align*}
	\inner{f}{\PSW f}{\mu_G^\tau}
	&= \inner{f}{T R T^* f}{\mu_G^\tau}
	= \inner{RT^*f}{RT^* f}{\nu_G^\tau} \ge 0
	\end{align*}
	for any $f: \Omega_G^\tau \to \R$. 
	Hence, we derive 
	$$
	\GAP(\PSW) \ge \gamma_{\mathrm{min}}  \cdot \GAP(\mathcal B_{\mathcal D}).
	$$
	\cref{lemma:sw-3,lemma:sw-4} imply that $\gamma_{\mathrm{min}} = \exp(-O(\mathrm{vol}(\mathcal D)))$ and the result follows. 
\end{proof}

\section{Random-cluster dynamics}
\label{sec:rc}

In this section we establish our result for the random-cluster dynamics on the wired tree.
In particular, we prove \Cref{thm:rc:intro} from the introduction.
We formally define first the \emph{random-cluster model} on a finite graph $G = (V(G),E(G))$.
Configurations of the random-cluster model are subsets of $E(G)$ 
and we use $\RC(G)$ for the set of all random-cluster configurations of $G$.
A boundary condition 
for the random-cluster model on $G$ is a partition $\xi = \{\xi_1,\xi_2,\dots\}$ 
of the boundary $\partial G \subseteq V(G)$ of $G$, such that all vertices in each $\xi_i$ are always in the same connected component of any configuration.
(We can think of the vertices in $\xi_i$ as being connected to a single external vertex.)
Given parameters $p \in (0, 1)$, $q > 0$ and a boundary condition $\xi$, 
the random-cluster distribution assigns to each $A \in \RC(G)$ a probability given by 
\begin{equation}\label{eq:rc}
\pi^\xi_G(A)=\frac1{Z_{\textsc{rc}}(G)}\,p^{|A|}(1-p)^{|E(G) \setminus A|}q^{c^\xi(A)},
\end{equation}
where $c^\xi(A)$ is the number of connected components in $(V(G),A)$,
taking into account the connections given by $\xi$; 
$Z_{\textsc{rc}}(G)$ is the corresponding partition function.

In the case when $G$ is the finite complete $d$-ary tree $T=(V(T),E(T))$, we set
the boundary condition on the external boundary $\partial T$.
The case where every element of $\xi$ is a single vertex corresponds to the free boundary condition (i.e., no external connections). 
In this setting, $\pi_T^\xi$ becomes the independent bond percolation on $T$ with parameter $\frac{p}{q(1-p)+p}$; see~\cite{Haggstrom}.
We shall focus on the case of the wired boundary condition, where $\xi = \{\partial T\}$. 
We adopt the common notation $\xi = 1$ for this case and denote by $\pi^1$ the random-cluster distribution on $T$ with wired boundary.
We note that 
$\pi^1$ has quite interesting phase transitions including three different critical thresholds; see~\cite{Haggstrom,J} for more details.

We consider two standard Markov chains for the random-cluster model.
This is the standard Markov chain that transitions from a random-cluster configuration $A_t \subseteq E(T)$ to a new configuration $A_{t+1}\subseteq E(T)$ as follows:
\begin{enumerate}
	\item choose an edge $e\in E$ uniformly at random;
	\item let $A_{t+1} = A_t \cup \{e\}$ with probability
	$$
	\frac{\pi^1(A_t \cup \{e\})}{\pi^(A_t \cup \{e\})+\pi^(A_t \setminus \{e\})} = \left\{\begin{array}{ll}
	\frac{p}{q(1-p)+p} & \mbox{if $e$ is a ``cut-edge'' in $(V,A_t)$;} \\
	p & \mbox{otherwise;}
	\end{array}\right.
	$$
	\item otherwise, let $A_{t+1} = A_t \setminus \{e\}$;
\end{enumerate}
the edge $e$ is a {\it cut-edge} in $(V(T),A_t)$ if the number of connected components in $A_t\cup \{e\}$ and $A_t \setminus \{e\}$ differ.
We prove the following result concerning the mixing time of edge heat-bath dynamics on the wired tree.

\begin{lemma}
	\label{lem:rc:glauber}
	For all integer $q\geq 2$, all $p \in (0,1)$, and all $d\geq 3$, for the random-cluster model on an $n$-vertex
	complete $d$-ary tree with {\em wired boundary condition},
	the mixing time of the heat-bath Glauber dynamics is $O(n \log{n})$.
\end{lemma}	

We also consider a natural variant of the SW dynamics on random-cluster configurations reversible with respect to $\pi^1$ defined as follows.
Given a random-cluster configuration $A_t \in \RC(T)$, the next configuration $A_{t+1}$ is obtained by:
\begin{enumerate}
	\item For each connected component $C$ in $(V(T \cup \partial T),A_t)$ with wired boundary condition (i.e., components containing a vertex from $\partial T$ are regarded as a single component), 
	we a pick a new spin from $\{1,\dots,q\}$ u.a.r.\ and assign it to every vertex of $C$; this gives a spin configuration $\sigma_t \in [q]^{T \cup \partial T}$;
	\item Obtain $A_{t+1} \subseteq E(T \cup \partial T)$ by including each monochromatic edge of $E(T \cup \partial T)$ in $\sigma_t$ independently with probability $p$. 
\end{enumerate}

We also prove that the mixing time of the random-cluster SW dynamics is $O(\log n)$ for all $p$ and all integer $q \ge 2$.

\begin{lemma}
	\label{lem:rc:sw}
	For all integer $q\geq 2$, all $p \in (0,1)$, and all $d\geq 3$, for the random-cluster model on an $n$-vertex
	complete $d$-ary tree with {\em wired boundary condition},
	the mixing time of the SW dynamics is $O(\log{n})$.
\end{lemma}

Note that Lemmas~\ref{lem:rc:glauber} and~\ref{lem:rc:sw} imply~\Cref{thm:rc:intro} from the introduction.

Our goal is to study random-cluster dynamics under the wired boundary condition,
so it will be convenient for us to consider the graph $\hat T = (V(\hat T),E(\hat T))$
that results from adding an external vertex $w$ to $T$ connected to every vertex in $\partial T$. 
Let $W$ be the set of edges between $w$ and $\partial T$.
Let us consider the joint Edwards-Sokal measure on $\hat T$; see~\eqref{eq:joint}. Specifically we are interested in the conditional measure $\nu := \nu_{\hat T}(\cdot \mid \sigma_w=1,A(W)=1)$
and its spin and edge marginals: $\mu_{\hat T}$ and $\pi_{\hat T}$, respectively.
The spin marginal $\mu_{\hat T}$ is supported on the set 
of configurations $\sigma \in [q]^{V(\hat T)}$
where $\sigma(\partial T \cup \{w\}) = 1$; let $\hat \Omega$ denote this set.

Let $\mu^1$ be the Potts distribution on $T$ with ``all 1'' boundary condition
on $\partial T$ and 
let $\Omega^1$ be the support of $\mu^1$;
note that $\Omega^1 \times \{1\} = \hat\Omega$.

We provide first the proof of \Cref{lem:rc:sw}.

\begin{proof}[Proof of \Cref{lem:rc:sw}]
	By~\cref{thm:EM-EOF} and the fact that EM holds at all temperatures 
	under a monochromatic boundary condition, 
	we have that for every function $f: \Omega^1 \to \R_{\ge 0}$ 
	\begin{equation}\label{eq:ent:spin}
	\Ent_{\mu^1}(f) \le C_{\textsc{eo}} \left( \mu^1[\Ent_E(f)] + \mu^1[\Ent_O(f)] \right)
	\end{equation}
	where $C_{\textsc{eo}}$ is independent of $n$. 
	From this, it follows that 
	any function $\hat f: \hat \Omega \to \R_{\ge 0}$
	\begin{equation}\label{eq:ent:spin:marg}
	\Ent_{\mu_{\hat T}}(\hat f) \le C_{\textsc{eo}} \left( \mu_{\hat T}[\Ent_{\hat E}(\hat f)] + \mu_{\hat T}[\Ent_{\hat O}(\hat f)] \right),
	\end{equation}
	where $\hat E \subseteq \hat T$ is the set of all even vertices of $\hat T$ and $\hat O = \hat T \setminus \hat E$ is the set of odd vertices. 
	
	Let $\JCH$ denote the joint spin-edge configurations on $\hat T$.
	From~\cite{BCPSVnew}, we know that the even-odd factorization entropy of the spin marginal in~\eqref{eq:ent:spin:marg} implies 
	that for all functions $\hat g: \JCH \to \R_+$,
	\begin{align}\label{entfact2o-intro}
	\Ent_{\nu} (g)\leq C \,\left(\nu\left[\Ent_{\nu} (g\mid\sigma)\right] +\nu\left[\Ent_{\nu} (g \mid A)\right]\right),
	\end{align}
	where $C = C(C_{\textsc{eo}},d,\beta) \ge 1$.
	
	We proceed now as in~\cite{BCPSVnew},
	where the consequences of the above spin-edge factorization of entropy for dynamics of the random-cluster model are explored.
	We consider first the SW dynamics in the joint space $\JCH \times \JCH$. 
	Let $K$ denote the $\JCH \times \JCH$ stochastic matrix corresponding
	to re-sampling the spins of a joint configuration given the edges;
	let $Q$ be the stochastic matrix corresponding to re-sampling the edges given the spins. 
	That is, 
	\begin{align*}
	K ((\sigma,A),(\tau,B)) &= \1(A=B) \nu (\tau \mid A ) \\
	Q ((\sigma,A),(\tau,B)) &= \1(\sigma=\tau) \nu (B \mid \sigma ).
	\end{align*}
	
	Let $P = \frac{K+Q}{2}$ and let $g: \JCH \to \R_+$ such that $\nu[g] = 1$.
	The  convexity of $x \log x$ implies 
	\begin{align}
	\label{eq:convex}
	Pg\log (Pg)\leq \frac12 Kg\log (K g) + \frac12 Qf\log (Q g).
	\end{align}
	Since $\nu[g]=1$, we have $\nu[Pg]=\nu[Kg]=\nu[Qg]=1$.
	Taking expectations with respect to $\nu$ in \eqref{eq:convex} we obtain
	\begin{align}
	\label{convex-exp}
	\Ent_\nu( Pg )\leq \frac12\left[ \Ent_\nu(Kg)+\Ent_\nu(Qg)\right].
	\end{align}
	The following identities follow from standard decompositions of the entropy functional:
	\begin{align*}
	\Ent_\nu(g) &= \Ent_\nu(\nu[g|A]) + \nu[\Ent_\nu(g|A)] ;\\
	\Ent_\nu(g) &= \Ent_\nu(\nu[g|\sigma]) + \nu[\Ent_\nu(g|\sigma)].
	\end{align*}
	Noting that $Kg (\sigma,A)= \nu(g\mid A)$ and $Qg (\sigma,A)= \nu(g\mid\sigma)$,
	we obtain
	\begin{align*}
	\Ent_\nu(g)=\Ent_\nu(Kg)+\nu\left[\Ent_\nu(g\mid A)\right]= \Ent_\nu(Qg)+\nu\left[\Ent_\nu(g\mid \sigma)\right].
	\end{align*}
	Then, \eqref{convex-exp} becomes
	\begin{align}\label{convj2}
	\Ent_\nu( Pg ) \leq \Ent_\nu(g) - \frac12\nu\left[\Ent_\nu(g\mid A)+ \Ent_\nu(g\mid \sigma)\right],
	\end{align}
	and~\eqref{entfact2o-intro} implies $\Ent_\nu( Pg)\leq (1-\delta)\Ent_\nu(g)$
	with $\delta=1/2C$. From Lemma~5.1 in~\cite{BCPSVnew} it follows that
	\begin{align}
	\label{eq:ent:sw:cont}
	\Ent_\nu( KQg)\leq (1-\delta)\Ent_\nu(g).
	\end{align}
	
	Now, let $\PSWE$ denote the transition matrix of the SW dynamics for the random-cluster model on $\hat T$ (conditioning on $A(W)=1$); that is, the SW dynamics for the edge marginal of $\nu$. 
	It is straightforward to check that if the function $g$ depends only on the edge configuration, and $\tilde g$ is the projection of $g$ to the edges (i.e., $g(\sigma,A) = \tilde g(A)$), we have $\PSWE \tilde g (A)=  KQg (\sigma,A).$
	Therefore, for any function $g\geq 0$ depending only on the edge configuration, and such that  $\nu[g] = \pi_{\hat T}[\tilde g]=1$, one has
	\begin{align*}
	\Ent_{\pi_{\hat T}}(\PSWE \tilde g) = \pi_{\hat T}[(\PSWE \tilde g) \log (\PSWE \tilde g) ]= \nu [(KQg)\log( KQ g)] = \Ent_\nu(KQg).
	\end{align*}
	From \eqref{eq:ent:sw:cont} we then get
	$$
	\Ent_{\pi_{\hat T}}(\PSWE \tilde g) \le (1-\delta)\Ent_\nu(g) = (1-\delta)\Ent_{\pi_{\hat T}}(\tilde g).  
	$$
	This show that $\PSWE$ contracts the entropy at a constant rate which implies that the mixing time of the SW dynamics for  is $O(\log n)$.
	The result follows by noting that the SW dynamics for $\pi_{\hat T}$
	and the SW dynamics for $\pi^1$ are the same Markov chain and thus
	have the same mixing time.
\end{proof}

We conclude with the proof of \Cref{lem:rc:glauber}.

\begin{proof}[Proof of \Cref{lem:rc:glauber}]
	We introduce the following auxiliary Glauber Markov chain, known as the \emph{single bond dynamics}. This chain was introduced in~\cite{Ullrich} and is quite useful for comparing random-cluster dynamics to Markov chains of the joint space. 
	In one step of the single bound dynamics, every connected component is assigned a spin from $[q]$ uniformly at random; a random edge $e$ is then chosen and if the endpoints of $e$ are monochromatic the edge is added to the configuration with probability $p$ and removed otherwise. The state of $e$ does not change if its endpoints are bi-chromatic.
	
	Let $\PSB$ denote the transition matrix of the single bond dynamics on $\hat T$, which is reversible with respect to $\pi_{\hat T}$; see~\cite{Ullrich}.
	Let $\Omega(\hat T)^1$ be the support of $\pi_{\hat T}$.
	For a function $g: \JCH \to \R_{\ge 0}$,
	let $\tilde g$ denote its projection to the edges, i.e., $g(\sigma, A) = \tilde g(A)$.
	The Dirichlet form associated to the single bond chain satisfies
	\begin{align}
	\label{eq:rc:df1}
	\D_{\PSB}(\tilde g,\tilde g)
	&= 
	\frac1{|E(\hat T)|}\sum_{e\in E(\hat T)}\nu\left[\Var_\nu(g \mid\sigma,A(E(\hat T) \setminus{e})\right];
	\end{align}
	see~(7.5) in~\cite{BCPSVnew}.
	
	Let $\PHB$ be the transition matrix of the heat-bath edge Glauber dynamics on $\hat T$.
	Since for every $A,B \subseteq E(\hat T)$, the transition probabilities $\PSB(A,B)$ and $\PSB(A,B)$ differ by at most a multiplicative factor that depends only on $q$ and $p$, the Dirichlet forms of these chain satisfy:
	\begin{align}
	\label{eq:rc:df2}
	c_1\D_{\PSB}( \tilde g, \tilde g)\leq \D_{\PHB}( \tilde g, \tilde g)\leq c_2\D_{\PSB}( \tilde g, \tilde g), 
	\end{align}
	for any function $\tilde g: \Omega(\hat T)^1\mapsto \R_{\ge 0}$ and suitable constants $c_1,c_2>0$.
	
	Combining~\eqref{eq:rc:df1} and ~\eqref{eq:rc:df2} we obtain:
	$$
	\D_{\PHB}( \tilde g, \tilde g) \ge \frac{c_1}{|E(\hat T)|}\sum_{e\in E(\hat T)}\nu\left[\Var_\nu(g \mid\sigma,A(E(\hat T) \setminus{e}))\right].
	$$
	Note that $\Var_\nu(g \mid\sigma,A(E(\hat T) \setminus{e}))$ is the variance of $g$ due to state of the edge $e$, which is Bernoulli random variable. 
	A well known relation between entropy and variance of Bernoulli random variable
	(see, e.g., Theorem A.1 and Corollary A.4 in \cite{DSC}) shows that, for all $g\geq 0$,
	\begin{equation*}
	\Ent_\nu(g \mid\sigma,A(E(\hat T) \setminus{e}))\leq c_3\Var_\nu(g \mid\sigma,A(E(\hat T) \setminus{e})),
	\end{equation*}
	where $c_3$ is a constant depending only on $q$, $p$ and $d$. Hence,
	\begin{equation}
	\label{eq:rc:df3}
	\D_{\PHB}( \tilde g, \tilde g) \ge \frac{c_1 c_3^{-1}}{|E(\hat T)|}\sum_{e\in E(\hat T)}\nu\left[\Ent_\nu(g \mid\sigma,A(E(\hat T) \setminus{e}))\right].
	\end{equation}
	
	By~\cref{thm:EM-EOF} and the fact that EM holds at all temperatures 
	under a monochromatic boundary condition, 
	the even-odd factorization of entropy in~\eqref{eq:ent:spin:marg} holds for every $p \in (0,1)$ and integer $q \ge 2$.	
	When the spin marginal of $\nu$ satisfies the even-odd factorization,
	the following factorization of entropy was established in (5.9)~\cite{BCPSVnew}
	\begin{align*}
	\Ent_\nu(g)\leq c_4
	\sum_{v\in V(\hat T)}\nu\left[\Ent_\nu(g\mid\sigma_{V\setminus\{v\}},A)\right]
	+c_4\sum_{e\in E(\hat T)}\nu\left[\Ent_\nu(g\mid\sigma,A(E(\hat T) \setminus{e}))\right],
	\end{align*}
	which implies that when $g$ is a function that depends only on the edges
	$$
	\Ent_{\pi_{\hat T}} (\tilde g) = \Ent_\nu(g) \leq c_4
	\sum_{e\in E(\hat T)}\nu\left[\Ent_\nu(g\mid\sigma,A(E(\hat T) \setminus{e}))\right].
	$$
	Plugging this bound into~\eqref{eq:rc:df3}, we obtain
	$$
	\D_{\PHB}( \tilde g, \tilde g) \ge \frac{c_1 c_3^{-1} c_4^{-1}}{|E(\hat T)|} 	\Ent_{\pi_{\hat T}} (\tilde g). 
	$$
	This is a log-Sobolev inequality for $\PHB$ with constant $O({|E(\hat T)|}^{-1}) = O(n^{-1})$, which implies that the mixing time of $\PHB$ is $O(n \log n).$
	The result follows by noting that the heat-bath edge dynamics 
	on $\hat T$ (conditioned on the edges in $W$ being present) corresponds to a lazy version the heat-bath dynamics on $T\cup\partial T$ with a wired boundary condition. The laziness of the chain corresponds to when edges of $W$ are selected,
	and thus the mixing time of the heat-bath edge dynamics on $(T\cup\partial T,E(T\cup\partial T))$ with wired boundary condition is $O(n \log n)$ as claimed.
\end{proof}

\section{A lower bound for the SW dynamics}
\label{sec:lower-bound}

In this section we establish an asymptotically tight lower bound for the mixing time of SW dynamics on the $d$-ary tree that holds for any inverse temperature $\beta$ and any boundary condition. In particular, we prove \Cref{thm:lb:intro} from the introduction.


The main idea of the proof is to adapt the lower bound framework of Hayes and Sinclair~\cite{HS} for the Glauber dynamics to the SW setting on the $d$-ary tree.  
Similar ideas were used recently to establish a tight lower bound for the mixing time of the SW dynamics on the grid~\cite{BCPSVstoc} in the high temperature region.

\bigskip\noindent\textbf{SW coupling.} \ We consider the following standard coupling of the steps of the SW dynamics. Consider two copies of the SW dynamics on $T = (V(T),E(T))$, the complete $d$-ary tree with $n$ vertices.
Recall that with a slight abuse of notation we also use $T$ for the vertex set of the tree and $E$ for the edge set.  
Let ${X_t}$ and ${Y_t}$ be the configurations of these copies at time $t \ge 0$.
We can couple the steps of the SW dynamics as follows:
\begin{enumerate}
	\item Draw 
	$|E|$ independent, uniform random numbers from $[0,1]$, one for each edge. Let $r_e(t) \in [0,1]$ denote the random number corresponding to the edge $e \in E$.
	\item Draw $|V|$ independent, uniform random numbers from $\{1,...,q\}$, one for each vertex. Let $s_v(t) \in \{1,...,q\}$ denote the random number for $v \in V$.
	\item Let $A_X = \{e \in M(X_t): r_e(t) \le p\}$ and $A_Y = \{e \in M(Y_t): r_e(t) \le p\}$. (Recall that $M(X_t)$ and $M(Y_t)$ denote the set of monochromatic edges in $X_t$ and $Y_t$, respectively.)
	\item For each connected component $\omega$ of $(V,A_X)$ or $(V,A_Y)$,
	we let $s_{\omega} = s_v(t)$, where $v$ is the highest vertex (i.e., closest to the root) in $\omega$.
	Then, every vertex of $\omega$ is assigned the spin~$s_{\omega} $.
\end{enumerate}

This coupling has the property that after assigning the edges, two identical connected components in $A_X$ and $A_Y$ will always be assigned the same spin.

\bigskip\noindent\textbf{A starting condition.} \
We let $B_1,B_2,\dots B_N \subset T$ be the collection of subtrees of $T$ rooted at level $L_k$ where $k = \lceil \frac 12 \log_d n \rceil$.
That is, $B_i = T_v$ for some $v \in L_k$ and hence $N = \Theta(\sqrt{n})$. 
(Recall that $L_k$ is the set of vertices of $T$ that are of distance exactly $k$ from $\partial T$.)

Let $B = \bigcup_{i=1}^N B_i$ and let $\partial B$ be the set of parents of the roots of the $B_i$'s. Let $e_i$ be an edge with both endpoints in $B_i$ that is at distance $1$ from $\partial T$. For definiteness, we take $e_i$ to be the ``left-most'' edge; note that 
one of the endpoints of $e_i$ corresponds to a leaf of $T$ (and $B_i$) and $e_i$ does not intersect $\partial T$.
Let $\mathcal A_i$ be the set of configurations on $B_i$
in which the spins at the endpoints of $e_i$ are the same.	

Let $\{X_t\}$ and $\{Y_t\}$ be two instances of the SW dynamics, coupled with the SW coupling from above. 
We let $\{X_t\}$ will be an instance of the SW dynamics on the tree $T$ with a suitably chosen starting condition. The chain $\{Y_t\}$ will have the same staring state as $\{X_t\}$ with moves 	restricted to $B$. 
For the initial state $X_0$ we set
the spins of all the vertices in $U = T \setminus B$ to $1$.
Then, 
the configuration 
in $B_i$ is sampled (independently) proportional to $\mu_{B_i}^1$ on $\mathcal A_i$; here $\mu_{B_i}^1$ denotes 
the Potts measure on $B_i$ with the spin $1$ as the fixed boundary condition on the parent of the root of $B_i$. (The boundary condition incident to the leaves of each $B_i$ is that from $\partial T$.)
$\{X_t\}$ is an instance of the standard SW dynamics on $(T,E)$.

As mentioned, $\{Y_t\}$ only updates the spins of the vertices in $B$. 
That is, after adding all the monochromatic edges independently with probability $p=1-\exp(-\beta)$, only the connected components fully contained in $B$ update its spin.
We set $Y_0 = X_0$ and couple the evolution of $Y_t$ and $X_t$ using the SW coupling defined above.
Note that $\{Y_t\}$ is a Markov chain on the configurations on $B$ with stationary distribution $\mu_B^1 = \otimes_{i=1}^N \mu_{B_i}^1$.
We also observe that a step of $\{Y_t\}$ is equivalent to performing one step of the SW dynamics in each $B_i$ independently.

Note that $X_0=Y_0$, and any disagreements between $X_t$ and $Y_t$ at later times~$t$
can arise only from the fact that $Y_t$ does not update the configuration outside~$B$: i.e., 
disagreements must propagate into each $B_i$ from its root.  
We provide the following bound on the speed of propagation of these disagreements
under the SW coupling.

\begin{lemma}
	\label{thm:dp}
	There exist sufficiently small constants $\xi > 0$ and $\alpha > 0$ such that
	the following is true.
	Let $\mathcal{\hat C} = \bigcup_{i=1}^{\hat R} e_i$ where $\hat R = n^\xi \le N$. Then,
	for $\tau = \alpha \ln n$ we have
	$$
	\Pr\left[X_\tau(\mathcal{\hat C}) = Y_\tau(\mathcal{\hat C})\right] = 1-o(1).
	$$
\end{lemma} 

\begin{proof}
	We bound the probability that $X_\tau(e_i) \neq Y_\tau(e_i)$ and then apply a union bound over the $e_i$'s.
	Let
	$v_i(t)$ be the vertex the configurations of $X_t(B_i)$ and $Y_t(B_i)$ disagree 
	that is the closest to the edge $e_i$ in the unique path from the root of $B_i$ to $e_i$.
	After one step of the SW coupling, 
	the increase in the depth of $v_{t+1}(i)$ relative to the depth of $v_t(i)$ is stochastically dominated
	by a geometric random variable, denoted $Z_t(i)$, with success probability $1-p$.
	This is because in order for disagreements to propagate from $v_t(i)$ towards $e_i$ a distance of $D$, we require a path of $D$ open edges starting $v_i(t)$; every edge in such path is either absent with probability $1$ (if its endpoints have different spins), or it is otherwise open with probability $p$.	
	Hence, the total distance disagreements can propagate from the root of $B_i$
	to $e_i$ in $\tau$ steps is dominated by $Z = \sum_{t=0}^\tau Z_t(i)$, where the $Z_t(i)$'s are independent random variables and $Z_t(i) \sim \mathrm{Geom(1-p)}$. Since $E[Z] = \tau/(1-p)$, applying a tail bound for the sum of geometric random variables (see, e.g., Theorem 2.1 in \cite{Janson}), we have for any $\delta \ge 1$
	$$
	\Pr[X_\tau(e_i) \neq Y_\tau(e_i)] \le \Pr\left[Z \ge \frac{(1+\delta) \tau}{1-p}\right] \le \exp\left(-\frac{\delta \tau}{4}\right) = \frac{1}{n^{\frac{\delta\alpha}{4}}}.
	$$
	Thus,
	the disagreements will not reach $e_i$ with probability at least $1-1/n^{\frac{\delta\alpha}{4}}$ provided $\frac{(1+\delta) \tau}{1-p} \le \frac 12 \log_d n$, which is true for sufficiently small $\alpha$.
	The result then follows from a union bound over the $\hat R = n^\xi$ 
	blocks for a suitably small $\xi > 0$.
\end{proof}

Another key ingredient of the proof is the following discrete time version of the completely monotone decreasing (CMD) property of reversible Markov chains from~\cite{HS}. This particular variant of the CMD property was already established in~\cite{BCPSVstoc}.

\begin{lemma}\label{lem:HS}
	Let $\{X_t\}$ denote a discrete time Markov chain with finite state space $\Omega$, reversible with respect to $\pi$ and with a positive semidefinite transition matrix. 
	Let $B\subset \Omega$ denote an event. 
	If $X_0$ is sampled proportional to $\pi$ on $B$, then 	$\Pr(X_t\in B)\geq \pi(B)$ for all $t \ge 0$, and for all $t \ge 1$ 
	\begin{equation*}\label{eq:bo1}
	\Pr(X_t\in B)\geq \pi(B) + (1-\pi(B))^{-t+1}(\Pr(X_1\in B)-\pi(B))^{t}.
	\end{equation*}
\end{lemma}

Finally, we also require the following decay property for the Potts measure on trees.

\begin{lemma}
	\label{lemma:decay}
	Let $(T,E)$ denote the rooted $d$-ary tree of height $h$.
	Let $\eta$ be an arbitrary boundary condition on $\partial T$,
	and for $i \in [q]$ let $\mu_T^{\eta,i}$ denote the Potts distribution
	on $T$ with boundary condition $\eta$ on $\partial T$ and with the parent of the root of $T$ fixed to spin $i$.
	Let $e$ be the ``left-most'' edge with both endpoints in $T$ at distance $1$ from $\partial T$, and let $\mu_{e}^{\eta,i}$ denote the marginal of $\mu_T^{\eta,i}$
	on the edge $e$.
	Then, for any fixed $\beta > 0$, $q \ge 2$ and $d \ge 3$, there exists a constant $c > 0$ such that for any $i,j \in [q]$
	\begin{equation}\label{eq:decay_prop}
	{\|\mu_{e}^{\eta,i}- \mu_{e}^{\eta,j}\|}_\textsc{tv} \le \frac{1}{e^{c h}}.
	\end{equation}
\end{lemma} 

\begin{proof}
	We prove the lemma by induction. When $h = 1$, there exists a constant $c_1 > 0$ such that ${\|\mu_{e}^{\eta,i}- \mu_{e}^{\eta,j}\|}_\textsc{tv} \le e^{-c_1}$ for any $i,j \in [q]$. 
	Now suppose that \eqref{eq:decay_prop} holds for $h-1$ for some positive constant $c \le c_1$ which we will specify soon. 
	Denote the root of $T$ by $r$ and let $v$ be the child of $r$ that is in the unique path from $r$ to $e$. 
	Denote the subtree rooted at $v$ by $T_v$. 
	Consider the following coupling of $\mu_{e}^{\eta,i}$ and $\mu_{e}^{\eta,j}$: first we optimally couple the spins $X_r$ and $Y_r$ at root $r$ from the marginals $\mu_r^{\eta,i}$ and $\mu_r^{\eta,j}$ respectively, then we optimally couple the configurations $X_e$ and $Y_e$ on $e$ given $X_r$ and $Y_r$. 
	Note that, there exists some constant $c_2 > 0$ such that $\Pr(X_r \neq Y_r) = {\|\mu_r^{\eta,i}- \mu_r^{\eta,j}\|}_\textsc{tv} \le e^{-c_2}$. 
	Meanwhile, since $X_e$ (resp., $Y_e$) is independent of the parent of $r$ given $X_r$ (resp., $Y_r$), we get
	\[
	\Pr(X_e \neq Y_e \mid X_r \neq Y_r) \le \max_{i,j \in[q]} {\|\mu_{e}^{\tilde{\eta},i}- \mu_{e}^{\tilde{\eta},j}\|}_\textsc{tv} \le e^{-c (h-1)}
	\]
	where $\mu_{e}^{\tilde{\eta},i}$ denotes the marginal of $\mu_{T_v}^{\tilde{\eta},i}$ on $e$ with $\tilde{\eta}$ being the boundary condition of the subtree $T_v$ induced from $\eta$ and $i$ being the spin of $r$, and the last inequality follows from the induction hypothesis. 
	We will set $c = \min\{c_1,c_2\}$. 
	It then follows that
	\begin{align*}
	{\|\mu_{e}^{\eta,i}- \mu_{e}^{\eta,j}\|}_\textsc{tv} &\le \Pr(X_e \neq Y_e) \\
	&= \Pr(X_r \neq Y_r) \Pr(X_e \neq Y_e \mid X_r \neq Y_r) \le e^{-c_2} \cdot e^{-c(h-1)} \le e^{-ch}.
	\end{align*}
	This establishes \eqref{eq:decay_prop} and proves the lemma. 
\end{proof}

We now proceed with the proof of~\cref{thm:lb:intro}.

\begin{proof}[Proof of~\cref{thm:lb:intro}]
	We would like to show that $
	{\|X_\tau - \mu\|}_\textsc{tv} > 1/2$ at some time $\tau = \Theta(\log n)$.	
	This implies that the mixing time of the SW dynamics is $\Omega(\log n)$.
	(Note that with a slight abuse of notation we use $X_\tau$ for the distribution of the chain at time $\tau$.)
	
	Let $\mathcal C = \bigcup_{i=1}^R e_i$, where $R \le \hat R$ will be chosen later ($\hat R$ is the constant from~\cref{thm:dp}).
	Let $\hat\mu_{\mathcal C}$ and $\hat\mu_{\mathcal C}^1$ be the marginals of $\mu$ and $\mu^{1}_B$, respectively, on $\mathcal C$. 
	Then,
	\begin{align}
	{\|X_\tau - \mu\|}_\textsc{tv} 
	&\ge {\|X_\tau(\mathcal C) - \hat\mu_{\mathcal C}\|}_\textsc{tv} \notag\\
	&\ge {\|Y_\tau(\mathcal C) - \hat\mu_{\mathcal C}\|}_\textsc{tv} - {\|X_\tau(\mathcal C) - Y_\tau(\mathcal C)\|}_\textsc{tv} \notag\\
	&\ge {\|Y_\tau(\mathcal C) - \hat\mu_{\mathcal C}^1\|}_\textsc{tv} - {\|\hat\mu_{\mathcal C}^1- \hat\mu_{\mathcal C}\|}_\textsc{tv} -
	{\|X_\tau(\mathcal C) - Y_\tau(\mathcal C)\|}_\textsc{tv}. \label{eq:main-trig}
	\end{align}
	Our goal is to find a lower bound for the first term of~\eqref{eq:main-trig}
	and upper bounds for the other two.
	First, since we choose $R \le \hat R$, Theorem~\ref{thm:dp} implies that
	\begin{align*}
	{\|X_\tau(\mathcal C) - Y_\tau(\mathcal C)\|}_\textsc{tv} 
	\le \Pr(X_\tau(\mathcal C) \neq Y_\tau(\mathcal C)) 
	& = o(1).
	\end{align*}
	
	We proceed next to bound the term ${\|\hat\mu_{\mathcal C}^1- \hat\mu_{\mathcal C}\|}_\textsc{tv}$ in~\eqref{eq:main-trig}.
	The bound can be deduce straightforwardly from~\cref{lemma:decay}.
	Let $\Omega(A)$ be the set of all possible configurations on the set $A \subseteq T$.
	For a configuration $\psi$ on $U$, let $\hat\mu_{\mathcal C}^\psi$ denote the marginal of $\mu_B^\psi$ on $\mathcal C$.
	Here, we use $\mu_B^\psi$ to denote the Gibbs measure on $B$ with
	the configuration on $\partial B$ fixed by $\psi$ and the configuration on $\partial T$ fixed by the boundary condition on $\partial T$.
	
	Let $\hat\mu_{e_i}^1$, $\hat\mu_{e_i}^\psi$ be the marginals of $\hat\mu_{B_i}^1$, $\hat\mu_{B_i}^\psi$ on $e_i$, respectively. Then,
	\begin{align*}
	{\|\hat\mu_{\mathcal C}^1- \hat\mu_{\mathcal C}\|}_\textsc{tv} 
	&\le \sum_{\psi \in \Omega(U)} \mu(\psi) 	{\|\hat\mu_{\mathcal C}^1- \hat\mu_{\mathcal C}^\psi\|}_\textsc{tv}  
	\le  \sum_{\psi \in \Omega(U)} \sum_{i=1}^R \mu(\psi) {\|\hat\mu_{e_i}^1- \hat\mu_{e_i}^\psi\|}_\textsc{tv} 
	\le \frac{R}{e^{\frac c2 \log_d n}},
	\end{align*}
	where the second inequality follows from the fact that $\mu_B^1$ and $\mu_B^\psi$ are product measures, and the last one follows from~\cref{lemma:decay} and the fact that the height of each $B_i$ is $\frac 12 \log_d n$.
	Then, setting $R = \min\{\hat R, n^{\frac{c}{4 \ln d}},n^{\frac 14}\}$, we obtain
	that ${\|\hat\mu_{\mathcal C}^1- \hat\mu_{\mathcal C}\|}_\textsc{tv} = o(1)$.
	
	We turn our attention to the lower bound for the term ${\|Y_T(\mathcal C) - \hat\mu_{\mathcal C}^1\|}_\textsc{tv}$ in~\eqref{eq:main-trig}.
	For a configuration $\sigma$ on $B$, let $f(\sigma)$ denote the number of edges $e_i \in \mathcal C$ that are monochromatic in $\sigma$.
	For any $a \ge 0$
	we have 
	\begin{equation}
	\label{eq:lb:3}
	{\|Y_\tau(\mathcal C) - \hat\mu_{\mathcal C}^1\|}_\textsc{tv} \ge \Pr[f(Y_\tau) \ge a] - \Pr\nolimits_{\sigma \sim \mu^1_{B}}[f(\sigma) \ge a].
	\end{equation}
	We show that, for a suitable $\tau$ and any $i=1,\dots,R$,
	\begin{equation}
	\label{eq:lb:key}
	\Pr[Y_\tau(B_i) \in \mathcal A_i] \ge \mu^1_{B_i}(\mathcal A_i) + \frac{1}{R^{1/4}}.
	\end{equation}
	Assuming this is the case and setting $\mathcal W = \sum_{i=1}^R \mu^1_{B_i}(\mathcal A_i)$,
	by Hoeffding's inequality 
	$$
	\Pr\left[f(Y_\tau) \ge \mathcal W + R^{3/4} - \sqrt{R \log R}\right] \ge 1 - \frac{1}{R^2}
	$$
	and
	$$
	\Pr\nolimits_{\sigma \sim \mu^1_{B}}\left[f(\sigma) \ge \mathcal W + \sqrt{R \log R}\right] \le \frac{1}{R^2},
	$$
	which yields from~\eqref{eq:lb:3} that
	$
	{\|Y_\tau(\mathcal C) - \hat\mu_{\mathcal C}^1\|}_\textsc{tv} \ge 1 - {2}/{R^2}
	$
	by taking $a = \mathcal W + \sqrt{R \log R}$.
	
	To establish~\eqref{eq:lb:key}, note that by Lemma~\ref{lem:HS} 
	\begin{equation}
	\label{eq:lb:kb}
	\Pr(Y_\tau(B_i)  \in \mathcal A_i) \ge \mu_{B_i}^1(\mathcal A_i) + (1-\mu_{B_i}^1(\mathcal A_i))^{-\tau+1}(\Pr(Y_1(B_i)\in \mathcal A_i)-\mu_{B_i}^1(\mathcal A_i))^{\tau}.
	\end{equation}
	($\{Y_t\}$ is positive semidefinite since it is a product of SW dynamics in each $B_i$, and the SW dynamics is positive semidefinite~\cite{BCSV}.) 
	
	Let $\PSWI$ denote the transition matrix of the SW dynamics on $B_i$.
	Then
	\begin{align}
	\Pr(Y_1(B_i)\in \mathcal A_i) &= \sum_{\sigma \in \mathcal A_i} \frac{\mu_{B_i}^1(\sigma)}{\mu_{B_i}^1(\mathcal A_i)} \PSWI(\sigma,\mathcal A_i) = \sum_{\sigma \in \mathcal A_i} \frac{\mu_{B_i}^1(\sigma)}{\mu_{B_i}^1(\mathcal A_i)} \left(\theta(\sigma) + \frac{1-\theta(\sigma)}{q}\right) \notag\\
	&= \frac 1q + \frac{q-1}{q \mu_{B_i}^1(\mathcal A_i)} \sum_{\sigma \in \mathcal A_i} \mu_{B_i}^1(\sigma)\theta(\sigma)~\label{eq:lb:eq1},
	\end{align}
	where $\theta(\sigma)$ denotes the probability that, after the edge percolation phase of the SW step, the end points of the edge $e_i$ are connected in the edge configuration. In similar fashion, we obtain 
	\begin{align*}
	\mu_{B_i}^1(\mathcal A_i) &= \sum_{\sigma \in \Omega(B_i)} \mu_{B_i}^1(\sigma)  \PSWI(\sigma,\mathcal A_i) \\
	&= \sum_{\sigma \in \Omega(B_i)\setminus\mathcal A_i} \mu_{B_i}^1(\sigma) \PSWI(\sigma,\mathcal A_i) + \sum_{\sigma \in \mathcal A_i} \mu_{B_i}^1(\sigma) \PSWI(\sigma,\mathcal A_i) \notag\\
	&= \sum_{\sigma \in \Omega(B_i)\setminus\mathcal A_i} \frac{\mu_{B_i}^1(\sigma)}{q} +  \sum_{\sigma \in \mathcal A_i} {\mu_{B_i}^1(\sigma)}\left(\theta(\sigma) + \frac{1-\theta(\sigma)}{q}\right)\\
	&= \frac 1q + \frac{q-1}{q}\sum_{\sigma \in \mathcal A_i} \mu_{B_i}^1(\sigma) \theta(\sigma).
	\end{align*}
	Combining with~\eqref{eq:lb:eq1} we get
	\begin{align*}
	\Pr(Y_1(B_i)\in \mathcal A_i) - \mu_{B_i}^1(\mathcal A_i) 
	&= \frac{q-1}{q} \left(\frac{1}{\mu_{B_i}^1(\mathcal A_i)} - 1\right)\sum_{\sigma \in \mathcal A_i} \mu_{B_i}^1(\sigma) \theta(\sigma) \\
	&\ge  \frac{q-1}{q} \left(\frac{1}{\mu_{B_i}^1(\mathcal A_i)} - 1\right) p \cdot \mu_{B_i}^1(\mathcal A_i) 
	= \frac{q-1}{q} \left(1 - \mu_{B_i}^1(\mathcal A_i)\right) p,
	\end{align*}
	where in the last inequality we use the fact that $\theta(\sigma) \ge p$ when $\sigma \in \mathcal A_i$; recall that $ p = 1 - e^{-\beta}$.
	
	Plugging this bound into~\eqref{eq:lb:kb}, 
	\begin{align*}
	\Pr(Y_\tau(B_i)  \in \mathcal A_i) &\ge \mu_{B_i}^1(\mathcal A_i) + (1-\mu_{B_i}^1(\mathcal A_i))^{-\tau+1}\left(\frac{q-1}{q} \left(1 - \mu_{B_i}^1(\mathcal A_i)\right) p\right)^{\tau} \\
	&= \mu_{B_i}^1(\mathcal A_i) + (1-\mu_{B_i}^1(\mathcal A_i))\left(\frac{(q-1)p}{q} \right)^{\tau}
	\ge \mu_{B_i}^1(\mathcal A_i) + \frac{1}{R^{1/4}},
	\end{align*}
	where the last inequality holds for $\tau = \alpha \log n$ for small enough $\alpha > 0$
	since $\mu_{B_i}^1(\mathcal A_i) = \Omega(1)$.
\end{proof}

\section{Slow mixing for random-cluster dynamics}
\label{sec:rc:slow}

In this section, we prove Theorem~\ref{thm:rc:lb} from the introduction.
We show first how to embed a graph on the boundary of a tree using a random-cluster boundary condition.

\medskip\noindent\textbf{The embedding.} \ 
For a graph $G=(V_G,E_G)$, let $\G = (\hat V_G,\hat E_G)$ be the graph that results from replacing every edge $\{u,v\}$ of $G$ by
a path of length $2$. That is, $\{u,v\} \in E_G$ is replaced by $\{u,r_{uv}\}$ and $\{r_{uv},v\}$ where $r_{uv}$ is a new vertex.

Consider the tree $T_h$ of height $h$. For simplicity and without much loss of generality we assume that $T_h$ is a binary tree; in particular, our proofs extend immediately to the case when $T_h$ is a $d$-ary tree. For a given graph $G=(V_G,E_G)$, we show how to embed $\G$ in the boundary $T_h$ using a random-cluster boundary condition. 
Let $\ell > 0$ be a parameter for the embedding
and partition the $2^h$ leaves of $T_h$ into $2^{h-\ell}$ subsets each containing the $2^\ell$ leaves of a full binary subtree of $T_h$. Let $B_1,B_2,\dots,B_{2^{h-\ell}}$ denote these subtrees; see~\cref{fig:1}.

\begin{figure}
	
	\begin{center}
		\begin{tikzpicture}[scale=0.4]
		\draw (0,0) -- (8,0) -- (4,8) --cycle;
		\filldraw[fill=green] (0,0) -- (1,0) -- (0.5,1) --cycle;
		\filldraw[fill=green] (1.1,0) -- (2.1,0) -- (1.6,1) --cycle;
		\filldraw[fill=green] (2.2,0) -- (3.2,0) -- (2.7,1) --cycle;
		\filldraw[fill=green] (3.3,0) -- (4.3,0) -- (3.8,1) --cycle;
		\filldraw[fill=green] (7,0) -- (8,0) -- (7.5,1) --cycle;
		\put(60,4){\small{$\dots$}}
		\end{tikzpicture}
		\caption{The subtrees $B_1,B_2,\dots, B_{2^{h-\ell}}$ (in green) of $T_h$.}	\label{fig:1}
	\end{center}
\end{figure}
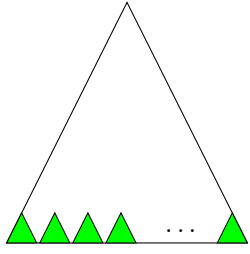	

We use exactly three edges from each subtree $B_i$ in the embedding.
Let $a_i,b_i$ be the two left-most leaves of $B_i$ and let $c_i$ be their parent.
We use the edges $\{a_i,c_i\}$ and $\{c_i,b_i\}$.
Specifically, every pair of edges $\{\{u,r_{uv}\},\{r_{uv},v\}\}$ of $\G$,
which corresponds to a single edge in $G$,
is mapped to a unique pair $\{a_i,c_i\}$, $\{c_i,b_i\}$; specifically, $u \rightarrow a_i$, $r_{uv} \rightarrow c_i$ and $v \rightarrow b_i$.
Observe that 
every middle vertex  $r_{uv}$ is mapped to exactly one $c_i$, whereas every other vertex 
is mapped to a number of vertices ($a_i$ or $b_i$) equal to its degree in $\G$.
We consider the boundary condition $\xi(\G)$ induced by this mapping.
That is, all the leaves of $T_h$ corresponding to the same vertex of $\G$ are wired together and all other leaves are left free.

\medskip\noindent\textbf{Random-cluster Markov chains.} \ For a graph $G = (V_G,E_G)$, let
$\pi = \pi_{G,p,q}$ denote the random-cluster distribution~\eqref{eq:rc} on $G$.
Let $\mathcal E = \{E_1,\cdots,E_k\}$ where $E_i \subseteq E_G$ and $\cup_{i=1}^k E_i = E_G$.
The heat-bath (random-cluster) block dynamics with respect to $\mathcal E$
is the Markov chain that 
from a random-cluster configuration $X_t$
picks $E_i$ u.a.r.\ from $\mathcal E$
and then updates the edge configuration on $E_i$ with a sample from the conditional distribution $\pi( \cdot \mid X_t(E_G\setminus E_i))$ to obtain $X_{t+1}$. The heat-bath \emph{edge} dynamics corresponds to the case where $\mathcal E = \{\{e\}: e \in E_G\}$. 

Consider the heat-bath block dynamics on $\G$ that in each step picks a ``middle'' vertex $r_{uv}$ u.a.r.\ and updates
the edges $\{\{u,r_{uv}\},\{r_{uv},v\}\}$ conditioned on the configuration off of these two edges. 
We denote this two-edge block chain by $\MC$ and let $\mathcal{M}$ be the heat-bath edge dynamics on the orginal graph $G$.
Our first observation is that the spectral gaps of $\MC$ and $\mathcal{M}$ are the same for a suitable choice of the edge parameters $p$ and $\p$.

\begin{lemma}
	\label{lemma:mg}
	$\GAP(\MC,\p,q) = \GAP(\mathcal{M},p,q)$, where $p = \frac{\p^2 }{\p^2  + 2\p(1-\p)+ (1-\p)^2 q}$.
\end{lemma}

\begin{proof}
	Consider the chain $\MC$ for the random-cluster model on $\G$ with parameters $\p$ and $q$.
	Suppose vertex $r_{uv}$ is selected. If the vertices $u$ and $v$ are connected in the configuration off of $\{\{u,r_{uv}\},\{r_{uv},v\}\}$, then the probability that both edges $\{u,r_{uv}\}$ and $\{r_{uv},v\}$ are both present after the update is 
	$$
	\frac{\p^2 q}{\p^2 q + 2\p(1-\p)q + (1-\p)^2 q^2} = \frac{\p^2 }{\p^2  + 2\p(1-\p)+ (1-\p)^2 q} = p.
	$$
	Otherwise, if $u$ and $v$ are not connected then the probability that the edges $\{u,r_{uv}\}$, $\{r_{uv},v\}$  are both present is
	$$
	\left(\frac{\p}{q(1-\p)+\p}\right)^2 = \frac{p}{q(1-p)+p}. 
	$$
	These are exactly the transitions of the single-edge heat-bath dynamics $\mathcal{M}$ on $G$.
	
	Let $\hat\Omega$ and $\Omega$ be the set of random-cluster configurations of $\G$ and $G$ respectively. 
	For $A \in \Omega$, let $\hat\Omega(A)$ be the set of configurations that project to $A$; that is, for $B \in \hat\Omega(A)$, $\{u,r_{uv}\}$,$\{r_{uv},v\} \in B$ iff $\{u,v\} \in A$.
	It can be checked that $\pi(A) = \sum_{B \in \hat\Omega(A)} \hat\pi(B)$, where $\hat\pi = \pi_{\hat G,\hat p,q}$ is the random-cluster measure on $\G$ with parameters $\p$ and $q$. From this it follows that $\df{\mathcal M}{f} = \df{\MC}{\hat f}$ and $\Var_\pi(f) = \Var_{\hat\pi}(\hat f)$, for any $f \in \R^{|\Omega|}$ with
	$\hat f \in \R^{|\hat \Omega|}$ defined as $\hat f (B ) = f(A)$ for $B \in \hat\Omega(A)$.
	The result follows.
\end{proof}

It will be convenient to also consider the following heat-bath block dynamics on the tree $T_h = (V(T_h),E(T_h))$, which we call the modified heat-bath dynamics (MHB). 
Let $L_h$ be the set of leaves of $T_h$ and let $I_h = V(T_h)\setminus L_h$.
Let $W_h$ be the set of vertices of $T_h$ used in the embedding; i.e., the set of $a_i,b_i,c_i$ vertices defined above.
Given a random-cluster configuration $X_{t}$, one step of the MHB chain is given by:
\begin{enumerate}
	\item Pick $v \in I_h$ uniformly at random;
	\item If $v \in W_h$, i.e., $v = c_i$ for some $i$, then perform a heat-bath update on the set of edges $\{\{a_i,c_i\},\{c_i,b_i\}\}$;
	\item Otherwise, perform a heat-bath update on $E(T_h) \setminus E(W_h)$.
\end{enumerate}

\subsection{Main result: transferring gap bounds from $G$ to $T_h$}

We relate next the spectral gap of the heat-bath edge dynamics $\mathcal{M}(T_h)$ on $T_h$ with boundary condition $\xi(\G)$ to the spectral gap of $\mathcal M$, the heat-bath edge dynamics on $G$.

\begin{theorem}\label{thm:lb:general} 
	Let $G=(V_G,E_G)$ be a graph and suppose there exist
	$q > 2$ and 
	$p \in (0,1)$ 
	such that 
	$
	\GAP(\mathcal M,p,q) \le \exp(-\Omega(|V_G|)).
	$
	Suppose that $p$ is such that $\frac 12 \ge \p \ge \frac{1}{e^{\gamma |V_G|}}$ for a sufficiently large constant $\gamma > 0$, and that for any small constant $\delta$
	there exists $\ell > 0$ such that 
	$\delta |V_G| > |E_G| \p^{\ell-1}$
	and  
	$|E_G| < 2^{h-\ell}$. Then,
	$$
	\GAP(\mathcal{M}(T_h),\p,q) \le  {e}^{-\Omega(|V_G|)}.
	$$
\end{theorem} 

\begin{remark}
	In order to use this theorem, the graph $G$ should be one for which slow mixing of the heat-bath edge dynamics is known for some value of $p$. The additional condition that  $|E_G| < 2^{h-\ell}$ is so that the graph can be embedded in $T_h$.
	The condition that for any small constant $\delta > 0$,  we have $\delta |V_G| > |E_G| \hat{p}^{\ell-1}$ is for controlling the influence between the different $\{a_i,b_i,c_i\}$ used in the embedding. 
\end{remark}	
\begin{remark}	
	If $n$ is the number of vertices of $T_h$, it follows from \cref{thm:lb:general} that:
	\begin{enumerate}
		\item If we take $G=K_a$ with $a = O(\sqrt{n})$, $p= \Theta(n^{-1/2})$ in the slow mixing window~\cite{BSmf,GSV}, then $\p = \Theta(n^{-1/4})$. Hence, for $\ell = 4$, we get an upper bound for the spectral gap of $\mathcal{M}(T_h)$ of $\exp(-\Omega(\sqrt{n}))$. Note that this bounds extends to the SW dynamics using the comparison results in~\cite{Ullrich}; this establishes Theorem~\ref{thm:lb:intro} from the introduction.
		
		\item 	We could also take $G=(V_G,E_G)$
		to be a random regular graph of maximum (constant) degree $d$
		with $|V_G| = \Theta(\sqrt{n})$. Then $p = \Theta(d^{-1})$ in the known slow mixing regime~\cite{GSVY} and so
		$\p = \Theta(d^{-1/2})$. Hence,
		taking $\ell$ to be a large enough constant, we again get an upper bound of $\exp(-\Omega(\sqrt{n}))$ for the spectral gaps of $\mathcal{M}(T_h)$ and of the SW dynamics.
	\end{enumerate}
\end{remark}

To prove Theorem~\ref{thm:lb:general} we extend the machinery from~\cite{BGV}.

\begin{proof}[Proof of Theorem~\ref{thm:lb:general}] 
	
	By assumption for $q>2$ and $p$, $\GAP (\mathcal M,p,q)\leq \exp(-\Omega(|V_G|))$.
	Then, by Lemma~\ref{lemma:mg},
	for the dynamics $\MC$ on $\G$ with parameters $\p$ and $q$, 
	we have  $\GAP (\MC,\p,q)\leq \exp(-\Omega(|V_G|))$.
	Hence, there must exist  $S_\star \subset \Omega_{\G}$ (the set of random-cluster configurations on $\G$) with $\pi_{\G,\p,q}(S_\star)\leq \frac 12$ such that 
	\begin{align}\label{eq:s-star-conductance:sketch}
	\Phi(S_\star) = \frac {Q_{\MC}(S_\star, S_\star^c)}{\pi_{\G}(S_\star)} \leq {e}^{-\Omega(|V_G|)}\,. 
	\end{align}
	Here $Q_\MC$ is the edge measure of $\MC$
	and $\pi_{\G} = \pi_{\G,\p,q}$.
	We will construct from this set $S_\star$, a set $A_\star \subset \Omega_{T_h}$ (the set of random-cluster configurations on $T_h$), such that :
	\begin{align}
	\label{eq:lb:conductance}
	\Phi(A_\star) = \frac{\QMHB(A_\star, A_\star^c)} {\pi^{\xi(\G)}({A_\star})} \leq {e}^{-\Omega(|V_G|)}\,, \qquad \mbox{and}\qquad \Phi(A_\star^c) =\frac{\QMHB(A_\star, A_\star^c)} {\pi^{\xi(\G)}({A_\star^c})} \leq {e}^{-\Omega(|V_G|)}\,,
	\end{align}
	where $\QMHB$ denotes the edge measure of the MHB dynamics on $T_h$
	with boundary condition $\xi(\G)$
	and $\pi^{\xi(\G)} = \pi^{\xi(\G)}_{\T_h,\p,q}$. 
	This upper bounds on the conductances of $A_\star$ and $A_\star^c $
	imply that, under the boundary condition $\xi(\hat G)$, $\GAP(MHB,\hat p,q) \leq \exp(-\Omega(|V_G|))$ (see, e.g., Theorem 13.14 in~\cite{LPW}). The result then follows by noting that 
	$
	\GAP(MHB,\hat p,q) \ge \GAP(\mathcal M(T_h),\hat p,q);
	$
	for a proof of this standard fact see Lemma 7.5 in~\cite{BGV}.

	Le $A_i = \{a_i,b_i,c_i\}$. 
	Let $\{\xi_1,\dots,\xi_k\}$ be the partition of $\{A_1,A_2,\dots\}$ induced by $\xi(\G)$.
	That is, $A_i$ and $A_j$ are in the same set of partition if they are wired in $\xi(\G)$.
	For a configuration $\omega$ on $E(V(T_h)\setminus W_h)$,
	we say that $\xi_i \stackrel{\omega}\longleftrightarrow \xi_j$
	if there is an open path in $\omega$ between vertices $c_x$ and $c_y$ 
	belonging to different sets of the partition $\{\xi_1,\dots,\xi_k\}$.
	Let 
	$$
	\mathcal{S}^{\xi(\G)}(\omega) = \{\xi_i \in \xi(G):\, \xi_i \stackrel{\omega}{\longleftrightarrow} \xi_j\text{  for some } j \neq i,~j \in \{1,\dots,k\}\}.
	$$
	For $M \ge 0$, let
	\begin{equation*} 
	\label{eq:def:r}
	\mathcal{R}^{\xi(\G)}(M) = \{\omega \in \{0,1\}^{E(V(T_h)\setminus W_h)}: |\mathcal{S}^{\xi(\G)}(\omega)| \leq  M \}\,.
	\end{equation*}
	In words, $\mathcal{R}^{\xi(\G)}(M)$ is the set of configurations on 
	$E(V(T_h)\setminus W_h)$ that connect at most $M$ elements of the partition $\{\xi_1,\dots,\xi_k\}$ by paths between the $c_i$'s. Any one of these connections ``distorts'' the embedding, so we want to avoid having too many of them. On the other hand, 
	we need a tail bound on the number of such connections, which improves as we increase $M$.
	
	Observe that any configuration $\theta$ on $E(W_h)$ corresponds to a configuration on $E_{\G}$. Let	
	\begin{align}\label{eq:A-M:sketch}
	A_{M} = \{\omega \in \Omega_{T_h}: \omega(E(W_h)) \in S_\star \,,\, \omega(E(V(T_h)\setminus W_h)) \in \mathcal{R}^{\xi(\G)}(M)\}\,.
	\end{align}  
	
	In words, $A_M$ is the set of configurations on $T_h$, whose configuration on $W_h$ is in the set of bad conductance for $\G$ and that allows at most $M$ bad connections distorting the embedding through the bulk of the tree.
	The following is the crucial estimate on the probability mass of the set $\mathcal{R}^{\eta}(M)$.
	
	\begin{claim}
		\label{claim:tail}
		Let $q\geq 1$.
		Let $\eta$ be any boundary condition for $T_h$
		and let $\Lambda$ be any configuration on $E(W_h)$. 
		Let $r = \p^{\ell-1}$ and $m = |E_G|$. Then,
		for every $M \ge 2mr$,
		\begin{align*}
		\pi_{T_h,\p,q}^{\eta} \big(\mathcal{R}^{\eta}(M) \mid \Lambda \big)  \ge 1-  e^{-\Omega(M )}.
		\end{align*}
	\end{claim} 
	
	From this, we can also deduce the following.
	
	\begin{claim}   
		\label{claim:am-prob-facts:sketch} Let $r = \p^{\ell-1}$ and $m = |E_G|$. If $M > mr$ we have:  
		\begin{enumerate}
			\item[(i)] $\pi^{\xi(\G)}(A_M) \ge q^{-M}({1-{e}^{-\Omega (M)}}) \pi_{\G}(S_\star)$\,;
			\item[(ii)] $\pi^{\xi(\G)} (A_M^c) \ge {e}^{-O(M)}$\,.
		\end{enumerate}
	\end{claim}

	With this bounds on hand, we can complete the ``lifting'' procedure.
	Let $\PMHB$ be the transition matrix for the MHB dynamics and for ease of notation set $L = E(W_h)$, $B = E(V(T_h)\setminus W_h)$, $\Omega = \Omega_{T_h}$ and $\xi = \xi(\G)$. We have 
	\begin{align}
	\QMHB(A_M,A_M^c)  &=  \sum_{\omega \in A_M} \sum_{\omega' \in A_M^c} \pi^\xi(\omega) \PMHB(\omega, \omega') \notag\\
	& \leq \sum_{\omega \in A_M} \sum_{\substack{\omega' \in \Omega: \\ \omega'(B) \notin \mathcal{R}^\xi(M)}} \pi^\xi(\omega) \PMHB(\omega, \omega') + \sum_{\omega \in A_M} \sum_{\substack{\omega' \in \Omega: \\ \omega'(L) \notin S_\star}} \pi^\xi(\omega) \PMHB(\omega, \omega')\,. \label{eq:edge-measure-decomp}
	\end{align}
	For the first term in (\ref{eq:edge-measure-decomp}), observe by definition of MHB dynamics, for every $\omega \in A_M$ and $M = \delta |V_G|$, for constant $\delta >0$ such that $M \ge 2|E_G|\p^{\ell-1}$,
	$$
	\sum_{\substack{\omega' \in \Omega: \\ \omega'(B) \notin \mathcal{R}^\xi(M)}} \PMHB(\omega,\omega') \le \sum_{\substack{\omega' \in \Omega: \\ \omega'(B) \notin \mathcal{R}^\xi(M)}} \pi^\xi\big(\omega'(B) \mid \omega(L)\big) \le {e}^{-\Omega(M)},
	$$
	where the last inequality follows from Claim~\ref{claim:tail}.
	Hence,
	$$
	\sum_{\omega \in A_M} \sum_{\substack{\omega' \in \Omega: \\ \omega'(B) \notin \mathcal{R}^\xi(M)}} \pi ^\xi(\omega) \PMHB(\omega, \omega') \le \pi^{\xi}(A_M)  {e}^{-\Omega(M)}.
	$$
	
	For the second term in (\ref{eq:edge-measure-decomp}), 
	observe that $\omega \neq \omega'$
	and that $\omega$ and $\omega'$ can differ in at most one gadget $A_i$ (two edges); otherwise  $\PMHB(\omega, \omega') = 0$. 
	With the assumption that $\p \le 1/2$, it follows that 
	$$\min_{\omega(L),\omega'(L)}{\MC}(\omega(L), \omega'(L)) \ge \frac{\p^2}{q^2|E_G|}.$$
	Hence,
	\begin{align*}
	\PMHB(\omega, \omega') &\le \frac{1}{|I_h|} \le  \frac{q^2|E_G|}{\p^2|I_h|} {\MC}(\omega (L), \omega' (L))\,.
	\end{align*}
	
	Then, since $|E_G|\leq |I_h|$  
	\begin{align*}
	\sum_{\omega \in A_M} \sum_{\substack{\omega' \in \Omega: \\ \omega'(L) \notin S_\star}} \pi^\xi (\omega) \PMHB(\omega, \omega') 
	&\le 	\frac{q^2}{\p^2}\sum_{\omega \in A_M} \sum_{\substack{\omega' \in \Omega: \\ \omega'(L) \notin S_\star}} \pi^\xi (\omega) \MC(\omega (L), \omega' (L)) \notag\\
	&\le \frac{q^2}{\p^2}\sum_{\theta \in \mathcal{R}^\xi(M)} \pi^\xi(\theta) \sum_{\omega_1  \in S_\star}\sum_{\omega_2\not\in S_\star} \pi^\xi(\omega_1 \mid \theta) \MC(\omega_1, \omega_2) \\ 
	& \le \frac{q^2}{\p^2} \pi^\xi(\mathcal{R}^\xi(M)) \sum_{\omega_1  \in S_\star}\sum_{\omega_2\not\in S_\star} \max_{\theta \in \mathcal{R}^\xi(M)} \pi^\xi(\omega_1 \mid \theta) \MC(\omega_1, \omega_2).
	\end{align*}
	
	Since there are at most $M$ wirings for any $\theta \in \mathcal{R}^\xi(M)$, we have the following standard inequality between random-cluster measures:	
	$$\max_{\theta \in \mathcal{R}^\xi(M)} \pi^\xi(\omega_1 \mid \theta) \le q^{2M} \pi_{\G}(\omega_1);$$
	see Lemma 2.2 in~\cite{BGV}. Therefore,
	\begin{align*}
	\sum_{\omega \in A_M} \sum_{\substack{\omega' \in \Omega: \\ \omega'(L) \notin S_\star}} \pi^\xi(\omega) \PMHB(\omega, \omega') 
	&\le \frac{q^{2M+2}}{\p^2}\sum_{\omega_1  \in S_\star}\sum_{\omega_2\not\in S_\star} \pi_{\G}(\omega_1) \MC(\omega_1, \omega_2) =\frac{q^{2M+2}}{\p^2}Q_{\MC}(S_\star,S_\star^c).
	\end{align*}

	Combining these two bounds, we get
	\begin{equation*}
	\QMHB(A_M, A_M^c) \le \pi^{\xi}(A_M)  {e}^{-\Omega(M)} + \frac{q^{2M+2}}{\p^2} Q_{\MC}(S_\star,S_\star^c).
	\end{equation*}
	
	Dividing by $\pi^{\xi}(A_M)$ and using the bounds from Claim \ref{claim:am-prob-facts:sketch}, we see that
	\begin{align*}
	\frac {\QMHB(A_M, A_M^c)}{\pi^{\xi}(A_M)} 
	\le {e}^{-\Omega(M )} + \frac{2q^{3M+2}}{\p^2} \frac{Q_{\M}(S_\star,S_\star^c)}{ \pi_{\G}(S_\star)} 
	\le  {e}^{-\Omega(M)}  + {e}^{O(M)}{e}^{-\Omega(|V_G|)}\,.
	\end{align*}
	for sufficiently large $M$, where the last inequality follows from (\ref{eq:s-star-conductance:sketch}) and the facts that $M = \delta |V_G|$ and $\p \ge \frac{1}{e^{\gamma|V_G|}}$ for a sufficiently large constant $\gamma > 0$. Similarly, we get
	\begin{align*}
	\frac {\QMHB(A_M, A_M^c)}{\pi^{\xi(G)}(A_M^c)} \le {e}^{-\Omega(M)}  + {e}^{O(M)}{e}^{-\Omega(|V_G|)}\,.
	\end{align*}
	Then, since $M = \delta |V_G|$, for some $\delta>0$ sufficiently small we obtain~\eqref{eq:lb:conductance}. 
\end{proof}

\begin{proof}[Proof of  Claim \ref{claim:tail}]
	Let $\mu = \pi_{T_h,\p,q}^{\eta} \big(\cdot \mid \Lambda \big)$.
	We have $\mu \prec \nu_{\p}$, where $\nu_{\p}$ is the independent percolation measure on $V(T_h)\setminus W_h$.
	
	Observe that for $c_i$ to be connected to any other $c_j$ through the bulk of $T_h$ there most be
	a path from $c_i$ to the root of the subtree $B_i$. 
	Since the edges of the events $\{c_i \leftrightarrow \textrm{root}(B_i)\}$ do not intersect for distinct $i$, these increasing events are independent under $\nu_{\p}$ and so for $X \sim Bin(m,r)$ with $r=\p^{\ell-1}$
	$$
	\mu(\mathcal{R}^{\eta}(M)^c) \le \Pr[X \ge  M].
	$$
	By a Chernoff bound
	$$
	\Pr[X \ge E[X]+\delta m] \le \exp\left(-m D(r+\delta || r)\right)
	$$
	where 
	$$
	D(r+\delta || r) = (r+\delta) \ln (\frac{r+\delta}{r}) + (1-r-\delta) \ln(\frac{1-r-\delta}{1-r}). 
	$$
	Setting $M = (r+\delta)m$, and using the fact that $m \ge M \ge 2rm$, it can be checked that $D(r+\delta || r) \ge (r+\delta)/20$ and the result follows.
\end{proof}

\begin{proof}[Proof of  Claim \ref{claim:am-prob-facts:sketch}]
	Recall that we set $L = E(W_h)$, $B = E(V(T_h)\setminus W_h)$ and $\xi = \xi(\G)$.
	For part (i), observe that if $\omega$ is sampled from $\pi^\xi$, then 
	\begin{align*}
	\pi^\xi(A_M) &  = \pi^\xi \big(\omega (L) \in S_\star \mid  \omega(B) \in \mathcal{R}^\xi(M)\big)  \pi^\xi \big(\omega(B) \in \mathcal{R}^\xi(M)\big)\,. 
	\end{align*}
	By Claim \ref{claim:tail},
	$$
	\pi^{\xi}(\omega(B) \in \mathcal{R}^\xi(M)) \ge 1-{e}^{-\Omega (M)}\,.
	$$
	Moreover, since 
	$$
	\pi^\xi(\omega (L) \in S_\star \mid  \omega(B) = 0) =  \pi_{\G}(S_\star)\,,
	$$
	it follows that
	$$
	\pi^{\xi}(\omega (L) \in S_\star \mid  \omega(B) \in \mathcal{R}^\xi(M)) \ge q^{-M} {\pi_{\G}(S_\star)}\,,
	$$
	and thus, 
	$$\pi^\xi(A_M)  \geq q^{-M}({1-{e}^{-\Omega (M)}}) \pi_{\G}(S_\star).$$ 
	Similarly for part (ii), we have 
	\begin{align*}
	\pi^\xi(A_M^c) 
	&\ge  \pi^\xi \big(\omega (L) \not\in S_\star \mid  \omega(B) \in \mathcal{R}^\xi(M)\big) 	\pi^\xi(\omega(B) \in \mathcal{R}^\xi(M)) \\
	&\ge  q^{-M} (1-{e}^{-\Omega(M)}) {\pi_{\G}(S_\star^c)}
	\end{align*}
	which is at least $e^{-O(M)}$ since $\pi_{\G}(S_\star) \leq \frac 12$.
\end{proof}

\end{document}